\documentclass[oneside,reqno]{amsart}
\usepackage{amssymb,amsmath}
\usepackage{amsfonts,amsthm}
\usepackage{xcolor}
\usepackage{amssymb}
\usepackage{amsfonts}
\usepackage[english]{babel}
\usepackage[utf8]{inputenc}

\usepackage{hyperref}
\hypersetup{
    colorlinks=true,
    linkcolor=blue,
    filecolor=magenta,      
    urlcolor=red,
}
\usepackage{textcomp}
\usepackage[all]{xy}
\usepackage{amsfonts,amsmath,verbatim,amssymb,amscd,latexsym}
\usepackage{amsthm,calc,enumerate, amscd}
\usepackage{graphicx}
\usepackage[T1]{fontenc}
\usepackage{pdfpages}
\usepackage{mathtools} 

\DeclarePairedDelimiter\floor{\lfloor}{\rfloor}
\newtheorem{theorem}{Theorem}[section]
\newtheorem{proposition}{Proposition}
\newtheorem{lemma}[theorem]{Lemma}

\newtheorem*{thm1.1}{\bf Theorem 1.1}
\newtheorem*{thm1.2}{\bf Theorem 1.2 (Unconditional)}
\newtheorem*{lem5.3}{\bf Lemma 5.3}

\theoremstyle{definition}

\newtheorem{corollary}{Corollary} 

\theoremstyle{remark}
\newtheorem{remark}[]{{ \bf Remark:}}

\begin{document}
\title{W\lowercase{eyl} \lowercase {bound for $p$-power twist of}  $GL(2)$ L-\lowercase{functions } }


\author{Ritabrata Munshi$^{(1)}$
  \and
  Saurabh Kumar Singh $^{(2)}$
} 

\address{$^{(1)}$ Current address : Stat-Math Unit,
Indian Statistical Institute, 
203 BT Road,  Kolkata-700108, INDIA.} 

\address{ Permanent address : School of mathematics,
Tata Institute of Fundamental Research, Homi Bhabha Road, 
Colaba,  Mumbai-400005, INDIA.}

\email{ritabratamunshi@gmail.com, \  rmunshi@math.tifr.res.in}

\address{$^{(2)}$ Current addres : Stat-Math Unit,
Indian Statistical Institute, 
203 BT Road,  Kolkata-700108, INDIA.} 

\email{skumar.bhu12@gmail.com}

\date{\today}


\begin{abstract}
Let $f$ be a cuspidal eigenform (holomorphic or Maass) on the full modular group $SL(2, \mathbb{Z})$  . Let $\chi$ be a primitive character of modulus $P$. We shall prove the following results:

\begin{enumerate}
\item Suppose $P = p^r$,  where $p$ is a  prime and $r\equiv  0 (\textrm{mod} \ 3)$.  Then we have 
\[
 L\left( f \otimes \chi, \frac{1}{2}\right) \ll_{f, \epsilon} P^{1/3 +\epsilon}, 
\] where $\epsilon > 0$ is any positive real number.

\item Suppose $\chi$ factorizes  as $\chi= \chi_1 \chi_2$, where $ \chi_i\textquotesingle s  $  are primitive character modulo $P_i$, where $P_i$ are primes, $P^{1/2 -\epsilon} \ll P_i \ll P^{1/2 + \epsilon}$   for $i=1,2$ and $P=P_1 P_2$. We have the Burgess bound
\[
 L\left( f \otimes \chi, \frac{1}{2}\right) \ll_{f, \epsilon}  P^{3/8 +\epsilon}, 
\] where $\epsilon > 0$ is any positive real number.
\end{enumerate}
\end{abstract} 
\maketitle 

\section{ Introduction }

Let $f $  be a holomorphic Hecke eigenform, or a Maass cusp form for the full modular group $ SL(2, \mathbb{Z})$. Let $\chi $ be a primitive character of modulus $P$. The twisted automorphic $L$-function of degree two associated to $(f, \chi)$ is defined as

%
 
\[
L(f\otimes \chi, s) := \sum_{n=1}^\infty \frac{\lambda_f(n) \chi (n) }{ n^s} ,
\] 
where $\Re s > 1$ and $\lambda_f(n)$ are normalised Fourier coefficients of $f$. These were studied by Hecke who proved that they are entire and they satisfy a functional equation relating $s$ to $1-s$.  We fix the form $f$ and vary the character $\chi$. One of the important problems in $L$-function theory is to provide an upper bound for $ L(f\otimes \chi, s)$ at the central point $s= 1/2$. The functional equation and the Phragmen-Lindel{\" o}f principle together with asymptotic of the Gamma functions  give us the convexity bound $ L(f\otimes \chi, 1/2) \ll_{f, \epsilon} P^{1/2+\epsilon}$. We shall prove the following results:

 \begin{theorem} \label{theorem 1} 
Let $f$ be a holomorphic eigenform of integral weight, or a weight zero Maass cusp form on the full modular group $SL(2, \mathbb{Z})$. Let $\chi$ be a primitive character of modulus $p^r$,  where $p$ is a  prime.  We have 
\[
 L\left( f \otimes \chi, \frac{1}{2} \right) \ll_{f, \epsilon} p^{\frac{1}{2} \left( r - \floor{r/3}\right) + \epsilon}, 
\] where $\epsilon > 0$ is any positive real number and $\floor{x} $ is the greatest integer less than or equal to  $x$.
\end{theorem}
\begin{corollary}
Let $f$ and $\chi$ be as above. Further suppose that $ r \equiv 0 ( \textrm{mod} \ 3)$. We have the following Weyl bound
\[
 L \left( f \otimes \chi, \frac{1}{2} \right) \ll_{f, \varepsilon} p^{\frac{r}{3}  + \varepsilon }. 
\]
\end{corollary}

Using the circle method, we shall also prove the following Burgess bound for $GL(2)$ $L$-functions for a particular case of modulus.  

\begin{theorem} \label{theorem 2}
Let $f$ be a holomorphic eigenform of integral weight, or a weight zero Maass form on the full modular group $SL(2, \mathbb{Z})$. Let $\chi$ be a primitive character of modulus $P$. Further assume that $\chi= \chi_1 \chi_2$, where $ \chi_i \textquotesingle s $ are primitive character modulo $P_i$, where $P_i \textquotesingle s $ are primes, $i=1,2$. We have 
\[
 L\left( f\otimes \chi, \frac{1}{2}\right) \ll_{f, \epsilon} P^{3/8 +\varepsilon}, 
\] where $\varepsilon > 0$ is any positive real number.
\end{theorem}
 Let us briefly recall the history of sub-convexity bounds for $L$-functions. The convexity bound for   the Riemann zeta function is given by $\zeta (1/2+ it)  \ll t^{1/4 + \varepsilon}$. For a Dirichlet $L$-function associated with a primitive Dirichlet character $\chi$ of modulus $q$, the convexity bound is given by  $L(1/2, \chi)\ll q^{1/4 + \varepsilon}$. Lindel{\" o}f hypothesis, which is a consequence of Riemann hypothesis, asserts that the exponent $1/4 + \varepsilon$ can be replaced by $\varepsilon$. Sub-convexity bound for $\zeta(s)$ was first proved by G. H. Hardy and J. E. Littlewood, based on the work of Weyl \cite{HW}. Establishing a bound for exponential sums, it has been proved that (see also \cite[page 99, Theorem 5.5]{ECT}) 
  \begin{equation} \label{weyl for gl1}
   \zeta(1/2 + it) \ll t^{1/6} \log^{3/2}t.
  \end{equation}
 It was first written down by E. Landau \cite{EL} in a slightly refined form, and has been generalised to all Dirichlet L-functions.  Since then it has been improved by several people. The best known result with exponent $ 13/ 84 \approx 0.15476 $ is  due to J. Bourgain \cite{BJ}.  In the other hand, $q$-aspect sub-convexity bound was first proved by D. A . Burgess \cite{DB}.  Using cancellation in character sums in short interval, he proved that 
 \begin{equation} \label{burger for gl2}
  L\left(  s, \chi \right) \ll_\epsilon q^{3/16 + \varepsilon}, 
  \end{equation}
 for fixed $s$ with $\Re s = 1/2$ and for any $\varepsilon >0$.  D. R. Heath-Brown \cite{HB} proved the  hybrid bound (bound in both parameters $q$ and $t$ together)  for Dirichlet $L$-functions. Recently introducing a theory to estimate exponential sums of the form
\[
\sum_{M < m \leq M + B} e \left( \frac{f(m)}{p^n}\right),
\]  where $f(t)$ is an analytic function on the ring of $p$-adic integers $\mathbb{Z}_p$, D. Mili{\' c}evi{\' c} \cite{DM} obtained the sub-Weyl exponent $0.1645$. More precisely, he proved the following theorem.

\begin{theorem} \label{theorem mili}
{\bf (D. Mili{\' c}evi{\' c} \cite{DM})}  Let $\theta> \theta_0 \approx 0.1645$ be given. There is an $r\geq 0$ such that 
\[
L(1/2, \chi) \ll p^r  q^\theta (\log q)^{1/2}
\] holds for any Dirichlet character $\chi$ to any prime power modulus $q=p^n$.
\end{theorem}
 From the above theorem we observe that we have a sub-convexity exponent which is less than  $1/6$ for a prime power modulus $q=p^n$ with $n\geq n_0$,  a sufficiently large number.  
    
 The $t$-aspect  Weyl exponent  for  $GL(2)$ $L$-functions is expected to be $1/3.$ For holomorphic forms,   this was first proved by Anton Good \cite{GOOD} using the spectral theory of automorphic functions. M. Jutila \cite{MJ} has given an alternative proof,   based only on the functional properties of $L(f, s)$ and $L(f\otimes \chi, s)$, where $\chi$ is an additive character. The arguments used in his proof were flexible enough to be adopted for the Maass cusp forms, as shown by   T. Meurman \cite{MERU1}, who proved the result for Maass cusp forms. 
 
However, the $q$-aspect sub-convexity bound seems to be a harder problem. It was first obtained by Duke–Friedlander–Iwaniec  using a new form of circle method. Assuming $\chi$ to be a primitive character of modulus $q$ and  $\Re s = 1/2$,  they obtained (see \cite[Theorem 1]{DFI}) $L(f\otimes \chi, s)\ll_f |s|^2 q^{5/11} \tau^2(q) \log q $, where $\tau(q)$ is the divisor function.  In the case of a general holomorphic cusp form, V. A. Bykovski{\v i} \cite{BYKO} used a trace formula expressing the mean values of the form to derive  a  stronger sub-convexity exponent $3/8$. G. Harcos used a similar method  given in \cite{DFI} to prove sub-convexity bound for Maass cusp forms as well. Refining the arguments used in \cite{BYKO}, V. Blomer and G. Harcos \cite{BH} obtained the Burgess exponent $ 3/8$ for a more general holomorphic or Maass cusp form. Till date, a proof of Weyl exponent $1/3$ exists only for quadratic characters, courtesy to the fundamental work of B. Conrey and  H. Iwaniec \cite{CI} and is not known for any other character. They proved that (see \cite[Corollary 1.2]{CI})
 \[
 L(f\otimes\chi, 1/2) \ll_{\varepsilon} q^{1/3 + \varepsilon},
 \] where $\chi$ is a quadratic character and $f$  is a primitive cusp form of weight $k\geq 12$ and level dividing $q$, or $f(z) =E(z, 1/2 +it) $,  Eisenstein series of full level. Extending the result of Theorem \ref{theorem mili} to $GL(2)$ $L$-functions, V. Blomer and D. Mili{\' c}evi{\' c} obtained the following theorem (see \cite[Theorem 2]{BM}). 
 \begin{theorem} \label{thm blomer mili}
 {\bf (V. Blomer and  D. Mili{\' c}evi{\' c} \cite{BM})} Let $p$ be an odd prime. Let $f$ be a holomorphic or Maass cuspidal newform for $SL(2, \mathbb{Z})$, and let $\chi$ be a primitive character of conductor $q=p^n$. Let $t\in \mathbb{R}.$ Then one has
\[
L( f \otimes \chi, 1/2+ it) \ll_{f, \varepsilon} (1+|t|)^{5/2} p^{7/6} q^{1/3 + \varepsilon}.
\] 
 
 \end{theorem}  From the above theorem, we can obtain the  sub-convexity  bound for $n> 7$ and improve the Burgess  exponent  as soon as $n>27$, and the exponent tends to Weyl exponent as $n\rightarrow \infty$.
In this paper we propose a different approach which produces an improvement over the known bounds.  In our Theorem \ref{theorem 1}, we are able to provide an improvement on the Burgess exponent as soon as $ n \geq 3 $, with the exception when $n=4, 8$ (in this case our exponent is same as Burgess exponent) and $n=5$. Main result of our Theorem \ref{theorem 1} is that we are able to achieve the Weyl exponent when $n \geq 3 $ and $ n \equiv 0 ( \textrm{mod} \ 3).$ Let us briefly explain the method of the proof.

\section{Sketch of the proof}
  We have the following general theorem for approximate functional equation of $L(f, s)$ (see \cite[page 98 Theorem 5.3]{IK})
\begin{theorem}
Let $G(u)$ be any function which is holomorhpic and bounded in the strip $-4 < \Re u< 4 $, even, and normalised by $G(0)=1$. Then for $s$ in the strip $ o\leq \Re s \leq 1$ we have

\begin{equation}
L(f, s)= \sum_n \frac{\lambda (n)}{n^s}  V_s \left( \frac{n}{X} \right) + \varepsilon(s, f) \sum_n \frac{\overline{\lambda } (n)}{n^{1-s} } V_{1-s} (nX),
\end{equation}
where $V_s(y)$ is a smooth function defined by
\[
V_s(y) = \frac{1}{2\pi i} \int_{(3)} y^{-u} G(u) \frac{\gamma(f, s+u)}{\gamma(f, s)} \frac{du}{u}
\] and 
\[
\varepsilon(f, s)= \varepsilon(f)  \frac{\gamma(f, 1-s)}{\gamma(f, s)}.
\]
\end{theorem}
Taking a dyadic subdivision of  approximate functional equation at $\Re s = 1/2$ and using properties of $V(y)$ (see \cite[ page 100 Proposition 5.4]{IK}), we have
 \begin{equation*}  
 L\left(\frac{1}{2} , f  \otimes \chi \right) \ll_{f, A} N^\varepsilon \sup_{N\leq P^{1+\varepsilon}} \frac{|S(N)|}{N^{1/2}} + P^{-A}, 
 \end{equation*} where $S(N)$ is a dyadic sum which is given by
\begin{equation} \label{SN}
S(N) := \sum_{n=1}^\infty \lambda_f (n)  \chi(n) V\left( \frac{n}{N} \right),
\end{equation} where $V(x)$ is a smooth bump function supported on the interval $[1,2]$ and satisfies $V^{(j)} (x) \ll_j 1.$  Trivially estimating, we obtain $S(N)\ll N^{1+\varepsilon}$. We shall examine $S(N)$ in following steps. For simplicity we assume that $q \sim Q$, $ N \sim p^r$ and $r \equiv 0(\textrm{mod} \ 3)$.

{\bf Step 1- Applying circle method: } 
 We shall apply Kloosterman's version of circle method(see Lemma \ref{circlemethod}) and conductor lowering mechanism introduced by first author \cite{RM0}. We obtain the sum of the form 
\begin{align*}
S (N) &=   2 \Re \int_0^1  
\mathop{\sum  \sideset{}{^\star}\sum}_{1\leq q\leq Q < q \leq q+Q}  \frac{1}{aq p^{\ell}}  \sum_{b ( p^{\ell})} \left\lbrace\sum_{n\sim N}  \lambda_f(n)   e\left(  \frac{ (\overline{a}+ bq) n}{p^{\ell} q} \right)    \right\rbrace  \\
& \hspace{2cm}  \times \left\lbrace\sum_{m\sim N}   \chi(m)  e\left( - \frac{ (\overline{a}+ bq) m}{p^{\ell} q} \right)  \right\rbrace dx. 
\end{align*} Trivially estimating after first step, we have $ S (N) \ll N^{2+ \varepsilon}$. 

{\bf Step 2- Applying Poisson summation formula:} In this step we shall apply Poisson summation formula to sum over $m$. The character $\chi$ has conductor $p^r$ and the additive conductor has a size $q$. Hence the total conductor for summation over $m$ has size $p^r q$. Initial sum over $m$ has size $N$ and we observe that dual summation is essentially supported on a summation of size $qp^r/N$. We observe that after application of the Poisson formula we are able to save $N/ \sqrt{p^r q}$ from sum over $m$. Evaluating the character sum, we also observe that $a (\textrm{mod} \ q)$ can be determined by a congruence relation. Total saving after the first step is given by 
\[
\frac{N}{\sqrt{p^r q}} \times \sqrt{ q} \sim p^{r/2}. 
\] Trivially estimating after the second step we obtain $S(N) \ll N p^{r/2}$. 

{\bf Step 3- Applying Voronoi summation formula:} We shall now apply Voronoi summation formula to sum over $n$, which  has  conductor of size $p^{\ell} q $. The dual length is essentially supported on a summation of size $p^{ 2\ell} q^2/ N$. We are able to save $N/ p^{\ell} q$ from Voronoi summation formula and $p^{\ell/2}$ by assuming square root cancellation in exponential sum over $b$. Total saving in third step is of size 
\[
\frac{N}{p^{\ell} q} \times p^{\ell/2}  \sim p^{r/2}. 
\] Trivially estimating after third step, we observe that $S(N)\ll p^{r + \varepsilon}$, which shows that we are on the boundary.  We are left with the sum of the form:

\begin{align*}
 S (N) &=      \sum_{n \ll p^{\ell} N^\varepsilon} \lambda_f( n) \left\lbrace \sum_{1\leq q\leq Q}  \sideset{}{^\star}\sum_{b ( p^{\ell})} \sum_{m \ll \frac{N^\varepsilon Q p^r}{N} }     \frac{\chi (q)}{ a q^2}  \overline{\chi} \left( m- (\overline{a} + bq ) p^{r- \ell}  \right)  e\left(- n\frac{\overline{a + bq} \  \overline{q}   }{p^{\ell}}  \right)   \right. \\
 & \left.  e\left(- n\frac{ p^r\overline{p^{2\ell}}  \overline{m}   }{q}  \right)  \mathcal{I} (x , q, m) \mathcal{J}(x, n, q)  \right\rbrace, 
\end{align*}  where the function $\mathcal{J}(x, n, q)$ is of size $O(1)$ but highly oscillatory.

{\bf Step 4- Cauchy inequality and Poisson summation formula:} To obtain additional saving, We shall  now apply Cauchy inequality to get rid of Fourier coefficients, but this process also squares the amount we need to save. We now open the absolute square and then interchange the summation over $n$. Applying the Poisson summation formula we are able to save $Q^2= p^{r-\ell}$ from the diagonal terms and $p^{\ell/2}$  from the non-diagonal terms. We observe that optimal choice for $\ell$ is given by $\ell= 2r/3$. Substituting the value of $\ell$ we obtain 
\[
S(N) \ll \frac{ p^{r+\varepsilon} }{Q} \ll p^{5r/6+\varepsilon}. 
\] In the following sections we shall provide  the proof of the theorem in detail.

\section{Preliminaries}
In this section we recall some basic facts about $SL(2, \mathbb{Z})$ automorphic forms (for details see \cite{HI} and \cite{IK}). Our requirement is minimal,  in fact Voronoi summation formula  and Rankin-Selberg bound (see Lemma \ref{rankin Selberg bound}) are all that we shall be using.
\subsection{Holomorphic cusp forms} 

Let $f $  be a holomorphic Hecke eigenform of weight $k$ for the full modular group $ SL(2, \mathbb{Z})$.  The   Fourier expansion of $f$ at the cusp $\infty$ is given by
\[
f(z)= \sum_{n=1}^\infty \lambda_f(n) n^{(k-1)/2} e(nz) \ \ \ \ ( \lambda_f(1) = 1 ), 
\]
where $ e(z) = e^{2\pi i z}$ and $\lambda_f(n), \ {n \in \mathbb{Z}}$ are the normalized Fourier coefficients.  It has been proved by Deligne that $|\lambda_f(n)| \leq d(n)$, where $d(n)$ is the divisor function. Let $\chi (n)$ be a primitive character of modulus $P$. The twisted $L$-function associated with form $f$ and character $\chi$ is given by
\[
L(f\otimes \chi, s) := \sum_{n=1}^\infty \frac{\lambda_f(n) \chi (n)}{n^s} \ =  \prod_p \left( 1 -\lambda_f(p) p^{-s} + \chi(p) p^{-2s} \right)^{-1} \ \ \ (\Re s>1). 
\]
 It has been studied by Hecke  who  proved that $L(f\otimes \chi, s)$ is an  entire function and satisfies the functional equation 
\[
\Lambda(f\otimes \chi, s) := \left( \frac{P}{ 2 \pi}\right)^s \Gamma \left( s +  \frac{ (k-1) }{ 2}  \right)   L(f\otimes \chi, s) = \varepsilon_\chi i^k \Lambda(f\otimes \overline{ \chi}, 1- s), 
\] where    $ |\varepsilon_\chi| = 1.$ From the functional equation and Phragmen-Lindel{\" o}f principle one can derive the convexity bound 
\[
L(f\otimes \chi, 1/2) \ll_{f, \epsilon} P^{1/2+\epsilon}.
\]

 We shall require the  following version of the Voronoi summation formula for holomorphic cusp form (for details, see appendix A.3 of \cite{KMV}).

\begin{lemma} \label{voronoi hol}
Let $\lambda_f(n)$ be as above. Let  $h$ be a compactly supported smooth function on the interval $(0, \infty)$. Let $q>0$ an integer and $a\in \mathbb{Z}$ are such that $(a, q)=1$. Then we have
\begin{equation} 
\sum_{n=1}^\infty \lambda_f (n) e_q(an) h(n) = \frac{1}{q}  \sum_{n=1}^\infty \lambda_f( n) e_q( - \overline{a}n) H\left( \frac{n}{q^2}\right),
\end{equation}
where $ a \overline{a} \equiv 1 (\textrm{mod} \  q)$, and 
\begin{align*} 
H (y)=  \int_0^\infty h(x)  J_{k-1} \left( 4\pi \sqrt{xy}\right) dx,
\end{align*} where $ J_{k-1} $  is Bessel function and $e_q(x)= e^{\frac{2 \pi i x}{q}}$.  

\end{lemma}

\subsection{Maass cusp forms} Let $f$ be a  weight zero Hecke-Maass cusp form with Laplace eigenvalue $1/4 + \nu^2$. The Fourier series expansion of $f$ at  $\infty$ is given by 
\[
f(z)= \sqrt{y} \sum_{n \neq 0} \lambda_f(n) K_{ i \nu} (2 \pi |n|y) e(nx). 
\] Let $\chi $ be a primitive Dirichlet character of modulus $P$. The twisted $L$-function is  defined by $ L(f\otimes \chi, s) := \sum_{n=1}^\infty \lambda_f(n) \chi (n)n^{-s}$ ( $Re s>1$). It extends to an entire function and satisfies the functional equation 
$ \Lambda(f\otimes \chi, s) = \varepsilon(f\otimes \chi ) \Lambda(f\otimes \chi, 1- s)$, where $ |\varepsilon(f\otimes \chi )| = 1$ and 
\[
\Lambda(f\otimes \chi, s) = \left( \frac{P}{\pi}\right)^s \Gamma \left( \frac{s  + i \nu   }{ 2}  \right)   \Gamma \left( \frac{s  - i \nu }{ 2} \right) L(f\otimes \chi, s) . 
\]  

We shall require the following Voronoi summation formula for the Maass form. This was first established by T.  Meurman \cite{MERU} for full level (for general case see appendix A.4 of \cite{KMV}).

\begin{lemma} \label{voronoi Maass}
{\bf Voronoi summation formula}
Let $h$ be a compactly supported smooth function in the interval $(0, \infty)$.  Let $\lambda_f(n)$ be the Fourier coefficient of weight zero Maass form for the full modular group $SL(2, \mathbb{Z})$, and $ a $, $q $ be positive integer with $(a, q) = 1.$ We have
\begin{equation} \label{varequation}
\sum_{n=1}^\infty \lambda_f (n) e_q(an) h(n) = \frac{1}{q} \sum_{\pm} \sum_{n=1}^\infty \lambda_f(\mp n) e_q(\pm \overline{a}n) H^{\pm} \left( \frac{n}{q^2}\right),
\end{equation}
where $ a \overline{a} \equiv 1 (\textrm{mod} \  q)$, and 
\begin{align*}
&H^{-} (y)= \frac{- \pi}{\cosh( \pi \nu)} \int_0^\infty h(x) \left\lbrace  Y_{2i\nu } + Y_{-2i\nu }\right\rbrace \left( 4\pi \sqrt{xy}\right) dx, \\
&H^{+} (y)= 4\cosh( \pi \nu) \int_0^\infty h(x)  K_{2i\nu }  \left( 4\pi \sqrt{xy}\right) dx,
\end{align*} where $Y_{2i\nu } $ and $ K_{2i\nu }$ are Bessel functions of first and second kind and $e_q(x)= e^{\frac{2 \pi i x}{q}}$.

\end{lemma} 
 \begin{remark}
 When $h$ is supported on the interval $[X, 2X]$ and satisfies $x^j h^{(j)} (x) \ll 1$, then  integrating by parts and using the properties of Bessel's function, it is easy to see that the sum on the right hand side of equation \eqref{varequation} are essentially supported on $n\ll_{f, \varepsilon} q^2 (qX)^\varepsilon/X$.  
 For smaller values of $n$ we will use the trivial bound that is $ H^{\pm} \left( \frac{n}{q^2}\right) \ll X$.  
 \end{remark}
 Next we record some lemmas which we shall use to estimate the sum $S(N)$.

\subsection{Some Lemmas}
We first recall the following Rankin-Selberg bound for Fourier coefficients:

\begin{lemma} \label{rankin Selberg bound}
Let $\lambda_f(n)$ be the normalised Fourier coefficients of a holomorphic cusp form, or of a Maass form. Then for any real number $x\geq 1$, we have 
\begin{align*}
\sum_{1\leq n \leq x} \left| \lambda_f(n) \right|^2 \ll_{f, \varepsilon} x^{1+\varepsilon}. 
\end{align*}

\end{lemma}

 Let $\delta: \mathbb{Z} \rightarrow \{ 0, 1 \}$ be the  Kronecker delta function, which is given by  
  
 \begin{equation} \label{deltan}
 \delta(n)= \begin{cases}
  1   \ \  \ \ \   \textrm{if }  \ \ \ \ \  n=0,\\ 
0   \ \ \ \     \textrm{otherwise} .
\end{cases}
\end{equation} 
We have the following lemma, which gives the Fourier-Kloosterman expansion of $ \delta(n) $ (see \cite[ page 470, Proposition 20.7]{IK}) 
\begin{lemma} \label{circlemethod}
Let $Q\geq 1$ be a real number. We have

\begin{equation} 
\delta(n)= 2 \Re \int_0^1 \mathop{\sum  \sideset{}{^\star}\sum}_{1\leq q\leq Q < q \leq q+Q} \frac{1}{aq} e\left( \frac{n \overline{a}}{q}- \frac{n x}{ aq}\right), 
\end{equation} 
where $\star$ restrict the summation by $(a, q) = 1$ and $ a \overline{a} \equiv 1 (\textrm{mod} \ q)$. 
\end{lemma}
We also use following Poisson summation formula(see \cite[page 69, Theorem 4.4]{IK}).

\begin{lemma} \label{poisson}
 {\bf Poisson summation formula:} $f:\mathbb{R } \rightarrow \mathbb{R}$ is any Schwarz class function. Fourier transform 
 of $f$ is defined  by
 \[
  \widehat{f}(y) = \int_{ \mathbb{R}} f( x) e(- x   y) dx,
 \] where $dx$ is the usual Lebesgue measure on $ \mathbb{R } $. 
 We have 
\begin{equation*}
 \sum_{ n \in \mathbb{Z}  }f(n) = \sum_{m \in \mathbb{Z} } \widehat{f}(m). 
\end{equation*} 
 
\end{lemma}

We also need to estimate the exponential integral of the form 
\begin{equation} \label{eintegral}
\mathfrak{I}= \int_a^b g(x) e(f(x)) dx,
\end{equation} where $f$ and $g$ are smooth real valued function on the interval $[a, b]$. Suppose we have $|f^\prime(x)| \geq B$, $|f^{(j)}(x)| \leq B^{1+\varepsilon}$ for $j\geq 2$  and $ |g^{(j)}(x)|\ll_j 1 $ on the interval $[a, b]$. Then by making the change of variable 
\[
u = f(x),  \ \ \ f^\prime(x) \ dx = du,
\] we have
 
\[
\mathfrak{I} = \int_{f(a)}^{f(b)} \frac{g(x)}{ f^\prime(x) } e(u) du \ \ \  ( x = f^{-1} (u)).
\] By applying integration by parts, differentiating $ g(x)/ f^\prime(x) $ $j$-times and integrating $e(u)$, we have
\begin{equation} \label{unstationary}
\mathfrak{I} \ll_{j, \varepsilon} B^{-j + \varepsilon}.
\end{equation}  This will be used at several places to show that certain exponential integrals are negligibly small in the absence of the stationary phase point. Next we consider the case of stationary phase point(i.e. point where derivative vanishes). 

\begin{lemma} \label{exponential inte}
Suppose $f$ and $g$ are smooth real valued functions on the interval $[a, b]$ satisfying

\begin{align} \label{huxely bound}
f^{(i)} \ll \frac{\Theta_f}{ \Omega_f^i}, \ \ g^{(j)} \ll \frac{1}{\Omega_g^j} \ \ \ \and \ \ \ f^{(2)} \gg \frac{\Theta_f}{ \Omega_f^2},
\end{align} for $i=1, 2$ and $j=0, 1, 2$. Suppose that $g(a) = g(b) = 0$. 
\begin{enumerate}
\item Suppose $f^\prime$ and $f^{\prime \prime}$ do not vanish on the interval $[a, b]$. Let $\Lambda = \min_{ x\in [a, b]} |f^\prime (x)| $. Then we have
\[
\mathfrak{I} \ll \frac{\Theta_f}{ \Omega_f^2 \Lambda^3} \left( 1 +\frac{\Omega_f}{\Omega_g} +\frac{\Omega_f^2}{\Omega_g^2} \frac{\Lambda}{\Theta_f/ \Omega_f} \right).
\]

\item Suppose that $f^\prime(x)$ changes sign from negative to positive at $x = x_0$ with $a<x_0 <b$. Let $\kappa= \min \{  b-x_0, x_0-a \}$. Further suppose that bound in equation \eqref{huxely bound} holds for $i=4$. Then we have the following  asymptotic expansion
\[
\mathfrak{I} = \frac{g(x_0) e( f(x_0)) + 1/8}{\sqrt{f^{\prime \prime } (x_0)}} + \left( \frac{\Omega_f^4}{ \Theta_f^2 \kappa^3} +  \frac{\Omega_f}{ \Theta_f^{3/2}  } +  \frac{\Omega_f^3}{ \Theta_f^{3/2} \Omega_g^2 }\right). 
\]
 
\end{enumerate}
\end{lemma}
\begin{proof}
See Theorem 1 and Theorem 2 of \cite{HUX}. 
\end{proof}

We also require to estimate exponential sum of the form 
\[
S(\chi, g,  p) = \sum_{x =1}^p \chi(g(x)), 
\] where $\chi $ is a multiplicative character modulo $p$ and $g(x)$ is a rational function over $\mathbb{Z}$. We record following general lemma due to Weil \cite{AW}. 

\begin{lemma} \label{weil}
Let $\chi $ be a multiplicative character modulo $p$. Let  $f(x)$ and $g(x)$ be  rational functions over $\mathbb{Z}$. Assume further that $f(x) \neq f_1^p(x) - f_1(x)$  for any $f_1 \in \overline{\mathbb{F}}_p [x]$ and  $g(x) \neq g_1^k(x)$ for any $g_1 \in \overline{\mathbb{F}}_p [x]$ where $k$ is the order of $\chi$. Then we have  
\begin{equation}
S(\chi, f,  g,  p) = \sum_{x =1}^p \chi(g(x)) e_p(f(x)) \leq  \left(n_1 + n_2 - 2+ \textrm{deg} (f)_\infty \right) \sqrt{p}, 
\end{equation} 
where $n_1$ is the number of poles or zeros of $g$, $n_2$ is the number of poles of $f$ and $(f )_\infty$ is the divisor $(f )  = \sum_{i=1}^{n_2} m_i P_i $ with $m_i$ being the multiplicity of the pole $P_i.$ 
\end{lemma}
 
 \section{Proof of Theorem \ref{theorem 1}   }
In order  to prove Theorem \ref{theorem 1}, we shall establish following bound. 

\begin{proposition} \label{proposition 1}
We have 
\begin{align*}
S(N) & \ll \begin{cases}
  N^{1+\varepsilon}    \ \  \ \ \  \ \ \ \ \ \ \   \textrm{if }  \ \ \ \   1\leq N \ll p^{\frac{2r}{3}+ \varepsilon} \\ 
  N^{1/2 } p^{\frac{1}{2} \left( r - \floor{r/3}\right) + \varepsilon}   \  \ \ \ \     \textrm{if }  \ \ \ \   p^{ \frac{2r}{3}}   \ll N \ll p^{r+\varepsilon}.
\end{cases}
\end{align*}

\end{proposition}

 \subsection{Application of circle method} 

We shall analyse the sum $S(N)$ for the  case  of Maass forms (holomorphic case is similar, even simpler). We first separate the oscillation of Fourier coefficients $\lambda_f(n)$ and $\chi(n)$ using delta symbol. We write $S(N)$ 

\[
 S(N) := \mathop{\sum \sum}_{ m, n=1 }^\infty \lambda_f (n)  \chi(m) \delta \left( n -m \right) V\left( \frac{n}{N} \right) V_1\left( \frac{m}{N} \right), 
 \] 
 where $V_1(y)$ is another smooth function, supported on the interval $[1/2, 3]$, $V_1(y) \equiv 1$ for $y \in [1,2]$ and satisfies $y^j V^{(j)} \ll _j 1.$ To analyse sum $S(N)$ we use the conductor lowering mechanism introduce by first author (see \cite{RM0} for discrete version of conductor lowering method and   \cite{RM} for the integral version). The integral equation $n = m$ is equivalent to the congruence $ n \equiv m (\textrm{mod} \ p^{\ell})$ and the integral equation $(n-m)/p^{\ell} = 0$, $\ell < r$. This process acts like a conductor lowering mechanism, as modulus $p^{\ell}$ is already present in the character $\chi$. We obtain

 \[
 S(N) := \mathop{\sum \sum}_{ \substack {{m, n=1} \\  p^{\ell} \mid (n-m) }}^\infty \lambda_f (n)  \chi(m) \delta \left( \frac{n -m}{p^{\ell}} \right) V\left( \frac{n}{N} \right) V_1\left( \frac{m}{N} \right), 
 \] 
Now using Lemma \ref{circlemethod}  for the expression of $\delta(n)$, we have 
 
\[
S(N)= S^+(N) + S^-(N),
\] with
\begin{align*} 
S^{ \pm} (N) &=   \int_0^1  
\mathop{\sum  \sideset{}{^\star}\sum}_{1\leq q\leq Q < q \leq q+Q}  \frac{1}{aq} \mathop{\sum \sum}_{ \substack {{m, n=1} \\  p^{\ell} \mid (n-m) }}^\infty  \lambda_f(n)  \chi(m) \notag\\
 & \hspace{1cm} \times e\left( \pm \frac{ \overline{a}(n-m)/p^{\ell} }{q} \mp \frac{x (n-m)/p^{\ell}  }{ aq}\right) V\left( \frac{n}{N} \right) V_1\left( \frac{m}{N} \right) dx. 
\end{align*}  We choose $Q=(N/p^{\ell})^{1/2} .$  We detect congruence relation $n \equiv m (\textrm{mod} \  p^{\ell})$ in the above expression using exponential sum. We have
 
\begin{align*}
S^{ \pm} (N) &=   \int_0^1  
\mathop{\sum  \sideset{}{^\star}\sum}_{1\leq q\leq Q < q \leq q+Q}  \frac{1}{aq p^{\ell}}  \sum_{b ( p^{\ell})}\mathop{\sum \sum}_{ \substack {{m, n=1} }}^\infty  \lambda_f(n) \ \chi(m) \notag\\
 & \hspace{1cm} \times e\left( \pm \frac{ (\overline{a}+ bq) (n-m)}{p^{\ell} q} \right) e\left( \mp \frac{ x (n-m)}{ a p^{\ell} q} \right) V\left( \frac{n}{N} \right) V_1\left( \frac{m}{N} \right) dx. 
\end{align*}  we will now  analyse the sum $S_2^{+}(N)$ ( analysis of $S_2^{-}(N)$ is just similar). We rearrange the sum as

 \begin{align} \label{s  plus in thm 1}
S^{ +} (N) &=   \int_0^1  
\mathop{\sum  \sideset{}{^\star}\sum}_{1\leq q\leq Q < q \leq q+Q}  \frac{1}{aq p^{\ell}}  \sum_{b ( p^{\ell})} \left\lbrace\sum_{n=1}^\infty  \lambda_f(n)   e\left(  \frac{ (\overline{a}+ bq) n}{p^{\ell} q} \right) e\left(  \frac{ x n}{  p^{\ell} a q} \right)   V\left( \frac{n}{N} \right) \right\rbrace \notag\\  
 & \hspace{2cm}  \times     \left\lbrace\sum_{m=1}^\infty   \chi(m)  e\left( - \frac{ (\overline{a}+ bq) m}{p^{\ell} q} \right) e\left(  \frac{ -m x }{  p^{\ell} a q} \right)  V_1\left( \frac{m}{N} \right) \right\rbrace dx.  
\end{align}

\subsection{Applying Poisson summation Formula}
We shall apply the Poisson summation formula to the sum over $m$ in equation  \eqref{s  plus in thm 1} as follows. Writing $m=\beta + c p^{r} q$, $c\in \mathbb{Z}$ and then applying the Poisson summation formula to sum over $c $, we have
\begin{align*}
& \sum_{m=1}^\infty   \chi(m)  e\left( - \frac{ (\overline{a}+ bq) m}{p^{\ell} q} \right) e\left(  \frac{ -m x }{  p^{\ell} a q} \right)  V_1\left( \frac{m}{N} \right) \\
 &  = \sum_{\beta(p^{r} q)} \chi(\beta) e\left( - \frac{ (\overline{a}+ bq) \beta}{p^{\ell} q} \right) \sum_{ c \in \mathbb{Z}}  V_1\left( \frac{ \beta +  c p^{r} q }{N} \right)   e\left(  \frac{ -(\beta +  c p^{r} q) x }{  p^{\ell} a q} \right) \\
  &  = \sum_{\beta(p^{r} q)} \chi_2(\beta) e\left( - \frac{ (\overline{a}+ bq) \beta}{p^{\ell} q} \right) \sum_{ m \in \mathbb{Z}} \int_{\mathbb{R}} V_1\left( \frac{ \beta + y p^{r} q }{N} \right)   e\left(  \frac{ -(\beta + y p^{r} q) x }{  p^{\ell} a q} \right) e(-my) dy. 
\end{align*} We now substitute the change of variable $ ( \beta + y p^{r} q)/N = z $ to obtain

\begin{align} \label{first poisson in theorem 1}
& = \frac{N}{p^{r} q}\sum_{ m \in \mathbb{Z}} \left\lbrace \sum_{\beta(p^{r} q)} \chi_2(\beta) e\left( - \frac{ (\overline{a}+ bq) \beta}{p^{\ell} q}  + \frac{m \beta}{ p^r q}\right) \right\rbrace \int_{\mathbb{R}} V_1 (y)  e\left(  \frac{ - N x y }{  p^{\ell} a q} \right) e\left(  \frac{ - N m y }{  p^{r} q} \right) dy  \notag \\
&:= \frac{N}{p^{r} q} \mathcal{C} (b , q) \mathcal{I} (x , q, m), 
\end{align} where $ \mathcal{C} (b , q)$ is the character sum and $\mathcal{I} (x , q, m) $ is the integral in the above expression. We now first evaluate the character sum in the following subsection.

\subsection{Evaluation of the character sum}

 Writing $q= p^{r_1} q^\prime$ with $(p, q^\prime)=1$, the character sum in equation \eqref{first poisson in theorem 1} can be written as 
\begin{align*}
\mathcal{C} (b , q)= \sum_{\beta (p^{r+ r_1} q^\prime)} \chi(\beta) e  \left( \frac{  - (\overline{a} + bq ) \beta }{ p^{\ell+r_1} q^\prime }  + \frac{ m \beta }{ p^{r+r_1} q^\prime  }\right).    
\end{align*} Writing $\beta= \alpha_1  q^\prime \overline{q^\prime} + \alpha_2  p^{r+r_1} \overline{ p^{r+r_1}}$ , the above character sum equals to

\begin{align*}
 \sum_{\alpha_1 (p^{r+ r_1})}  \chi(\alpha_1)  e\left( \frac{  - (\overline{a} + bq ) \alpha_1 \overline{q^\prime}}{ p^{\ell+r_1}  }  + \frac{ m \alpha_1 \overline{q^\prime} }{ p^{r+r_1}   }\right) \sum_{\alpha_2( q^\prime)} e  \left( \frac{  - (\overline{a} + bq ) \alpha_2  \overline{p^{r+r_1}} p^{r- \ell}}{  q^\prime }  + \frac{ m \alpha_2  \overline{p^{r+r_1}}}{  q^\prime  }\right). 
\end{align*} 
Again, writing $\alpha_1 = \beta_1 p^r + \beta_2$, where $\beta_2$ is modulo $ p^r$ and $ \beta_1$  modulo $ p^{r_1}$, we obtain
\begin{align*}
\mathcal{C} (b , q) &=  \sum_{\beta_2(p^r)} \chi(\beta_2)  e\left( \frac{  - (\overline{a} + bq ) p^{r-\ell}\beta_2 \overline{q^\prime }}{ p^{r+ r_1}  }  + \frac{ m \beta_2 \overline{q^\prime } }{ p^{r+ r_1}   }\right)  \sum_{\beta_1(p^{r_1})} e\left( \frac{  - (\overline{a} + bq ) \beta_1 \overline{q^\prime} p^{r - \ell}}{ p^{r_1}  }  + \frac{ m \beta_1 \overline{q^\prime} }{ p^{r_1}   }\right) \\
 & \hspace{2cm}  \times \sum_{\alpha_2( q^\prime)} e  \left( \frac{  - (\overline{a} + bq ) \alpha_2  \overline{p^{r+r_1}} p^{r- \ell}}{  q^\prime }  + \frac{ m \alpha_2  \overline{p^{r+r_1}}}{  q^\prime  }\right).
\end{align*} 
From the last two exponential sums, we obtain the congruence relations  $m-  \overline{a} p^{r- \ell} \equiv 0  (\textrm{mod}\ p^{r_1}  )$  and  $m-  \overline{a} p^{r- \ell} \equiv 0  (\textrm{mod}\ q^\prime  )$. Since we have $q = q^\prime p^{r_1}$, we obtain the congruence relation $m-  \overline{a} p^{r- \ell} \equiv 0  (\textrm{mod}\ q  )$, from which   $a (\textrm{mod} \ q)$  can be  determined. Sum over $\beta_2$ can be written as 

\begin{align*}
\sum_{\beta_2(p^r)} \chi(\beta_2)  e\left( \frac{\left( m - (\overline{a} + bq ) p^{r-\ell} \right)\beta_2 \overline{q^\prime }}{ p^{r+ r_1}  }  \right)  = \chi (q^\prime) \overline{\chi} \left( \frac{m- (\overline{a} + bq ) p^{r- \ell}}{p^{r_1}}\right) \sum_{\beta_2(p^r)}  \chi(\beta_2)  e\left( \frac{\beta_2 }{ p^{r}  }  \right) ,  
\end{align*}

  as $p^{r_1} \mid \left( m - (\overline{a} + bq ) p^{r-\ell} \right)$. We record this into following lemma.

\begin{lemma} Let $ \mathcal{C} (b , q)$ be as given in equation \eqref{first poisson in theorem 1}. we have
 
\begin{align*}
\mathcal{C} (b , q) & = \begin{cases}
   q \ \chi (q^\prime) \ \overline{\chi} \left( \frac{m- (\overline{a} + bq ) p^{r- \ell}}{p^{r_1}}\right)  \tau_{\chi}    \ \  \ \ \   \textrm{if }  \ \ \ \   m  \equiv \overline{a}  p^{r-\ell}  \mod  q \\ 
0   \ \ \ \     \textrm{otherwise} ,
\end{cases}
\end{align*}
where $q = q^\prime p^{r_1}$ and $ \tau_{\chi} $ denotes the Gauss sum. 
\end{lemma}
 For simplicity of notation we assume that $ q = q^\prime $($r_1 = 0$), as number of $r_1$ are bounded by $O \left(\log p^r \right)$. Next we consider the integral in equation \eqref{first poisson in theorem 1}. Integrating by parts $j$-times and using $ V_1^{(j)} (y) \ll 1$, we have
\begin{align*}
\mathcal{I} (x , q, m) \ll \left(\frac{Nx}{ p^{\ell} a q} + \frac{  N m  }{ p^r q} \right)^{-j}. 
\end{align*}  We observe that this integral is negligibly small if 
\begin{align*}
\left|  \frac{Nx}{ p^{\ell} a q} + \frac{  N m  }{ p^r q}\right| \gg N^{\varepsilon}. 
\end{align*} From the above inequality we obtain the effective range of $x$ as 

\begin{align} \label{range of x}
\left|  \frac{Nx}{ p^{\ell} a q} + \frac{  N m  }{ p^r q}\right| \ll N^{\varepsilon} \Rightarrow  \left|  x + \frac{   m a   }{ p^{r - \ell } }\right| \ll \frac{q}{Q},  
\end{align} as $a \asymp Q $ and $N/ p^{ \ell } = Q^2$. Again integrating by parts, taking $  V_1 (y)  e\left(  \frac{ - N x y }{  p^{\ell} a q} \right)$ as first function, we obtain 

\begin{align*}
\mathcal{I} (x , q, m) \ll  \left( 1 + \frac{Nx}{ p^{\ell} a q}  \right)^{j} \left( \frac{ p^r  q}{ Nm }  \right)^{j}. 
\end{align*} 
Hence the integral is negligibly small if $ m \gg ( p^r Q N^\varepsilon)/ N$. After first application of Poisson summation formula we obtain following expression for $ S^{ +} (N) $:

\begin{lemma}  Let $ S^{ +} (N)$ be as given in equation \eqref{s  plus in thm 1}.  We have 
 
 \begin{align} 
S^{ +} (N) &=    \int_{ x \ll q/Q} 
\sum_{1\leq q\leq Q}  \frac{1}{aq p^{\ell}}  \sum_{b ( p^{\ell})}  \left\lbrace\sum_{n=1}^\infty  \lambda_f(n)   e\left(  \frac{ (\overline{a}+ bq) n}{p^{\ell} q} \right) e\left(  \frac{ x n}{  p^{\ell} a q} \right)  V\left( \frac{n}{N} \right)  \right\rbrace    \notag\\  
 & \hspace{1cm}  \times \left\lbrace \frac{ \tau_{\chi}   \chi (q)  N }{ p^r }   \sum_{m \ll \frac{N^\varepsilon Q p^r}{N} }    \overline{\chi} \left( m- (\overline{a} + bq ) p^{r- \ell} \right) \mathcal{I} (x , q, m) \right\rbrace dx + O_A \left(p^{-A} \right),  
\end{align}   for any real $A >0$.

\end{lemma}
Estimating trivially at this stage, we have 
\begin{align*}
S^{ +} (N) & \ll \sum_{1\leq q\leq Q}  \frac{1}{aq p^{\ell}}  \sum_{b ( p^{\ell})}  \sum_{n=N}^{ 2N} \left| \lambda_f(n)  \right|   \ \ \times   \frac{  \left| \tau_{\chi} \right|   \chi (q)  N }{ p^r }   \sum_{m \ll \frac{N^\varepsilon Q p^r}{N} }       1  \\
& \ll \frac{1}{a p^{\ell}}  \   p^{\ell}  \ N  \ \frac{p^{r/2} N}{p^r}  \ \frac{Q p^r}{N} \ll N p^{r/2}. 
\end{align*}
Hence we are able to save $p^{r/2} $ from the first application of Poisson summation formula. 

\subsection{Applying Voronoi summation formula}

 We have $ (\overline{a}+ bq, q) = 1$. Given $a, $  there exists at most one $b$ mod $p^\ell$  such that $ \overline{a}+ bq  \equiv 0 (\textrm{mod} \  p^{\ell} )$.  For the rest of $b$ we apply  the  Voronoi summation formula to the sum over $n$ as follows (The case where $ \overline{a}+ bq \equiv 0 (\textrm{mod} \ p^{\ell})$ is similar and even simpler). We first write  $ \overline{a}+ bq = p^{\ell} q_1$, and then apply the Voronoi summation formula, which gives us more saving as the conductor is now smaller than $ qp^{\ell}$. Also we have saving of whole summation over $b$ modulo $ p^{\ell}$ ). We substitute $g(n)=  e(-nx/ p^{\ell} aq) V(n/N)$ in Lemma \ref{voronoi Maass} to get
\begin{align*}
\sum_{n=1}^\infty  \lambda_f(n)   e\left( \frac{ (\overline{a}+ bq) n}{p^{\ell} q} \right) e\left(  \frac{ x n}{  p^{\ell} a q} \right)  V\left( \frac{n}{N} \right) &  =   \frac{1}{p^{\ell} q}    \sum_{\pm} \sum_{n\geq 1} \lambda_f(\mp n) e\left( \pm  n \frac{ \overline{\overline{a}+ bq  } }{p^{\ell} q} \right)  \\
 & \hspace{1cm} H^{\pm} \left( \frac{ n}{q^2}, \frac{ N x}{aq}\right), 
\end{align*} where   
\begin{align*}
   H^{-}\left( \frac{ n}{q^2}, \frac{ N x}{ p^{\ell} aq}\right) =  \int_0^\infty e\left(- \frac{xy}{p^{\ell} a q} \right)  V \left( \frac{y}{N}\right) \left\lbrace  Y_{2 i \nu} + Y_{ - 2 i \nu}\right\rbrace \left( \frac{4 \pi \sqrt{ny}}{p^{\ell} q }\right)  dy 
\end{align*} (we have similar expression for  $  H^{+} (x, y)$).   Substituting $y/ N= z$, we have 

\begin{align} \label{math cal j x n q }
H^{-}   \left( \frac{ n}{q^2}, \frac{ N x}{P_1aq}\right) =   N \int_0^\infty e\left(- \frac{Nxy}{p^{\ell} a q} \right)   V \left( y\right) \left\lbrace  Y_{2 i \nu} + Y_{ - 2 i \nu}\right\rbrace \left( \frac{4 \pi \sqrt{n Ny}}{p^{\ell} q }\right)  dy : = N \ \mathcal{J} (x, n, q), 
\end{align}  
where $ \mathcal{J} (x, n, q)$ denotes the integral in above equation. Pulling out the oscillations, we have the following asymptotic formulae for Bessel functions(see \cite[Lemma C.2]{KMV}): 
 \begin{equation} \label{bessel}
 \ Y_{\pm2 i\nu}(x)  = e^{ix}U_{\pm2 i\nu}(x)  + e^{-ix} \overline{U}_{\pm2 i\nu}(x) \ \ \ \textrm{and} \ \ \  \left|x^{k} K_\nu^{(k)} (x) \right| \ll_{k , \nu} \frac{ e^{-x} (1+ \log |x|)}{(1+x)^{  1/2 }},
\end{equation}
where the function $ U_{\pm2 i\nu}(x)$ satisfies,
 
 \begin{equation} \label{parti der of bessel 2}
 x^j U_{\pm2 i\nu}^{(j)} (x) \ll_{j , \nu, k} (1+x)^{-1/2}.
 \end{equation}
 Also we have
$J_k (x) = e^{ix} W_k(x) + e^{-ix} \overline{W} (x)$, where 
 \[
 x^j  W_k^{(j)}(x) \ll_j \frac{1}{(1+x)^{1/2}}. 
 \] Substituting the above decomposition  for  $ Y_{\pm2 i\nu}(x)$ $ $, the first term of the integral in equation \eqref{math cal j x n q } is given by (estimation of second term is similar)

\begin{align} \label{mathcal J plus minus}
\mathcal{J}^{\pm} (x, n, q) := \int_0^\infty e^{ i \left( - \frac{ 2 \pi N x y}{p^{\ell} a q } \pm i \frac{4 \pi \sqrt{n Ny}}{p^{\ell} q }  \right) } V(y)  U_{2 i\nu}^\pm \left( \frac{4 \pi \sqrt{n Ny}}{p^{\ell} q } \right) dy, 
\end{align}
where we have denoted $U^+(y):= U(y)$ and $ U^-(y)= \overline{U}(y)$.  Integral $\mathcal{J}^{-} (x, n, q)$ has no stationary point. By equation \eqref{unstationary},  $\mathcal{J}^{-} (x, n, q)$ is negligibly small. For $\mathcal{J}^{+} (x, n, q)$ we apply the second statement of Lemma \ref{exponential inte} with 
\[
f(y) = - \frac{ 2 \pi N x y}{p^{\ell} a q } + i \frac{4 \pi \sqrt{n Ny}}{p^{\ell} q } \  \ \ \textrm{and} \ \ \ \ \  g(y) =  V(y) \   U_{2 i\nu} \left( \frac{4 \pi \sqrt{n Ny}}{p^{\ell} q } \right). 
\]
We have 
\begin{align*}
f^\prime(y)= - \frac{ 2 \pi N x }{p^{\ell} a q }  + \frac{2 \pi \sqrt{n N}}{ \sqrt{y} p^{\ell} q }, \ \ \ \  f^{\prime \prime}(y) = - \frac{ \pi \sqrt{n N}}{ y^{3/2} p^{\ell} q }.
\end{align*} We observe that 
\[
|F^{\prime \prime}(y_0)| \asymp  \frac{  \sqrt{n N}}{  p^{\ell} q }, 
\] where $ y_0$ is the stationary point,  which is $ y_0 \asymp 1$ as $ V(y)$ is supported on the interval $[1,2]$.  Using  $ U_{\pm2 i\nu} (x) \ll_{  \nu} (1+x)^{-1/2}$, and applying the second statement of the Lemma \ref{exponential inte}, we obtain 

\begin{equation} \label{bound of mathcal J x q n}
\mathcal{J}(x, n, q) \ll \frac{p^{\ell} q }{ \sqrt{n N}}, 
\end{equation} where $ \mathcal{J}(x, n, q) $ is given in equation  \eqref{math cal j x n q }. Also, integrating by parts we have
 
 \[
  \mathcal{J}(x, n, q) \ll_j \left( \frac{Nx}{p^{\ell} a q} + 1\right)^j \left( \frac{ p^{\ell}q}{\sqrt{nN}} \right)^j. 
 \] The integral is negligibly small if 
\[
\frac{p^{\ell} q }{ \sqrt{n N}} \ll p^{-\varepsilon} \Rightarrow n \gg p^\varepsilon p^{\ell}. 
\]

We record this result in the following lemma. After applying the Poisson and the Voronoi summation formula we have the following expression for $S^+(N)$. 

\begin{lemma} We have
\begin{align*} 
S^{ +} (N) &=    \int_{x\ll (qN^\varepsilon)/Q} 
\sum_{1\leq q\leq Q}  \frac{1}{q p^{\ell}}  \sideset{}{^\star}\sum_{b ( p^{\ell})}  
\frac{ \tau_{\chi}   \chi (q)  N }{ p^r }   \sum_{m \ll \frac{ Q p^{r} p^\varepsilon}{N} }     \frac{1}{a}\overline{\chi} \left( m- (\overline{a} + bq ) p^{r- \ell}  \right)  \mathcal{I} (x , q, m)   \notag\\  
 & \hspace{1cm}  \times \left\lbrace \frac{N}{ p^{\ell} q}    \sum_{\pm} \sum_{n \ll p^{\ell} p^\varepsilon} \lambda_f(\mp n) e\left( \pm  n \frac{ \overline{\overline{a}+ bq  } }{ p^{\ell}  q} \right) \mathcal{J}(x, n, q) \right\rbrace dx + O_A \left(p^{-A} \right). 
\end{align*}
\end{lemma} 

Estimating trivially, we have (assuming square-root cancellation in the character sum over $b$ and Lemma \ref{rankin Selberg bound}):

\begin{align*} 
S^{ +} (N) & \ll  \sum_{1\leq q\leq Q}  \frac{1}{q p^{\ell}}    
\frac{ \left| \tau_{\chi}  \right|   N }{ p^r }   \sum_{m \ll \frac{p^\varepsilon Q p^r}{N} }    \frac{\left|  \mathcal{I} (x , q, m)  \right| }{a} \times \frac{N}{ p^{\ell} q}   \sum_{n \ll p^{\ell} N^\varepsilon} \left| \lambda_f(\mp n) \mathcal{J}(x, n, q) \right| \\
& \hspace{2cm} \times \left| \sideset{}{^\star}\sum_{b ( p^{\ell})}  e\left( \pm  n \frac{ \overline{\overline{a}+ bq  } }{ p^{\ell}  q} \right)  \overline{\chi} \left( m- (\overline{a} + bq ) p^{r- \ell}  \right)  \right|  \\
& \ll  \frac{1}{a p^{\ell}}  \frac{  p^{r/2}   N }{ p^r } \ \frac{N^\varepsilon Q p^r}{N}  \times \frac{N}{ p^{\ell} q}   \sum_{n \ll p^{\ell} } \frac{p^{\ell} q }{ \sqrt{n N}}  \times p^{\ell/2} \ll  N^\varepsilon \sqrt{N}  p^{r/2} . 
\end{align*} This shows that we are on the boundary. To obtain additional saving, we shall now apply the Cauchy inequality to the summation over $n$ and then apply the Poisson summation formula. Interchanging the order of summation, we have
\begin{align} \label{S plus before cauchy}
 S^{ +} (N) &=   \frac{N^2  \tau_{\chi} }{ p^{ r + 2\ell}}     \sum_{n \ll p^{\ell} p^\varepsilon} \lambda_f( n) \hat{S_1} (n) + O_A \left(p^{-A} \right), 
\end{align} where
\begin{align*} 
\hat{S_1} (n) & =  \int_{x\ll q/Q} \sum_{1\leq q\leq Q}  \sideset{}{^\star}\sum_{b ( p^{\ell})} \sum_{m \ll \frac{p^\varepsilon Q p^r}{N} }     \frac{\chi (q)}{ a q^2}  \overline{\chi} \left( m- (\overline{a} + bq ) p^{r- \ell}  \right)  e\left(- n\frac{\overline{a + bq} \  \overline{q}   }{p^{\ell}}  \right)    \\
&  \hspace{1cm} \times    e\left(- n\frac{ p^r\overline{p^{2\ell}}  \overline{m}   }{q}  \right)  \mathcal{I} (x , q, m) \mathcal{J}(x, n, q) dx .  
\end{align*}

\subsection{Applying Cauchy Inequality}
We split the summation over $n$ into dyadic sum. Applying the  Cauchy inequality on the summation over $n$ in equation \eqref{S plus before cauchy} and using Lemma \ref{rankin Selberg bound}, we have

 \begin{align} \label{define S 2 hat L}
 S^{ +} (N) & \ll   \frac{N^2  |\tau_{\chi}| }{p^{ r + 2\ell}}    \sum_{\substack{ L \ll P_1 \\ L \textrm{-} \textrm{dyadic}}}  \left\lbrace\sum_{n\ll L} \left| \lambda_f(n)\right|^2  \right\rbrace^{1/2} \left\lbrace \sum_{n \in \mathbb{Z}} \left| \hat{S_1} (n)\right|^2 U\left(\frac{n}{L} \right) \right\rbrace^{1/2} \notag \\
 & \ll   \frac{N^2 | \tau_{\chi} | }{ p^{ r + 2\ell}}     \sum_{\substack{ L \ll P_1 \\ L \textrm{-} \textrm{dyadic}}}  L^{1/2} \left\lbrace \hat{S_2} (L) \right\rbrace^{1/2} ,
 \end{align}
where  $P_1 = p^{ell + \varepsilon} $ and $\hat{S_2} (L) $ is given by (opening the absolute square and pushing the  summation over $n$ inside):
\begin{align} \label{S 2 hat L}
\hat{S_2} (L) :=&  \int_{x\ll q/Q}  \int_{x^\prime \ll q^\prime/Q} \sum_{m \ll \frac{N^\varepsilon Q p^r}{N} }  \sum_{m^\prime \ll \frac{N^\varepsilon Q p^r}{N} }  \sum_{1\leq q \leq Q} \sum_{1\leq  q^\prime \leq Q} \frac{\chi(q)}{ a q^2} \frac{\overline{\chi} (q^\prime)}{ a^\prime q^{\prime 2} }  \sideset{}{^\star}\sum_{b( p^{\ell})} \sideset{}{^\star}\sum_{ b^\prime ( p^{\ell})}   \notag \\
  &  \hspace{10pt}\times  \overline{\chi} \left( m- (\overline{a} + bq ) p^{r- \ell}  \right) \chi \left( m^\prime- (\overline{a} + b^\prime q^\prime ) p^{r- \ell}  \right) \mathfrak{I} (x , q, m)    \mathfrak{I} (x , q^\prime, m^\prime)   \ \  \ \mathcal{T} \   \ dx \  dx^\prime,   
\end{align} where 

\begin{align} \label{mathcal T}
\mathcal{T} & := \sum_{n \in \mathbb{Z}}  e\left( n\left\lbrace -\frac{ p^r\overline{p^{2\ell}}  \overline{m}   }{q}   +  \frac{ p^r\overline{p^{2\ell}}  \overline{m^\prime}   }{q^\prime}  - \frac{\overline{a + bq} \  \overline{q}   }{p^{\ell}} + \frac{\overline{a + b^\prime q^\prime} \  \overline{q^\prime}   }{p^{\ell}} \right\rbrace\right) U\left(\frac{n}{L} \right)  \mathcal{J}(x, n, q)  \mathcal{J}(x, n, q^\prime) \notag\\  
\end{align}
\subsection{Second application of Poisson summation formula}

We write a smooth bump function $U(n/L) \mathcal{J}(x, n, q) \mathcal{J}(x, n, q^\prime) := U_1 (n/L)$, where $  \mathcal{J}(x, n, q)$ is  as given in equation \eqref{math cal j x n q }.  Writing $ n = \alpha + q q^\prime p^{\ell} c$, $c \in \mathbb{Z}$ and applying Poisson summation formula to sum over $c$, we have

\begin{align*}
\mathcal{T} & := \sum_{\alpha ( q q^\prime p^{\ell} )}  e\left( \alpha\left\lbrace -\frac{ p^r\overline{p^{2\ell}}  \overline{m}   }{q}   +  \frac{ p^r\overline{p^{2\ell}}  \overline{m^\prime}   }{q^\prime}  - \frac{\overline{a + bq} \  \overline{q}   }{p^{\ell}} + \frac{\overline{a + bq^\prime} \  \overline{q^\prime}   }{p^{\ell}} \right\rbrace\right)\\
 & \hspace{2cm} \int_{\mathbb{R}} \sum_{ n \in \mathbb{Z}}  U_1 \left( \frac{\alpha+ yq q^\prime p^{\ell} }{L}\right) e(-ny) dy. 
\end{align*} We now apply the change of variable $(\alpha+ yq q^\prime p^{\ell})/L = z$ to get 

\begin{align} \label{T in theorem 2}
\mathcal{T} & = \frac{  L }{ q q^\prime p^{\ell} }\sum_{ n \in \mathbb{Z}} \sum_{\alpha ( q q^\prime p^{\ell} )}  e\left( \alpha\left\lbrace -\frac{ p^r\overline{p^{2\ell}}  \overline{m}   }{q}   +  \frac{ p^r\overline{p^{2\ell}}  \overline{m^\prime}   }{q^\prime}  - \frac{\overline{a + bq} \  \overline{q}   }{p^{\ell}} + \frac{\overline{a + bq^\prime} \  \overline{q^\prime}   }{p^{\ell}} + \frac{n}{ q q^\prime p^{\ell}} \right\rbrace\right) \notag \\
 & \hspace{4cm} \times \int_{\mathbb{R}}   U_1 \left( y\right) e\left( - \frac{ n L y}{ q q^\prime p^{\ell} }\right) dy. 
\end{align} We have $U_1 \left( y\right)  =  U(y) \mathcal{J}(x, Ly, q) \mathcal{J}(x, Ly, q^\prime)$. From the expression of $ \mathcal{J}(x, Lu, q)$ in equation \eqref{math cal j x n q } (note that after change of variable we have $u \asymp 1$) and equation \eqref{parti der of bessel 2},  we have 

\begin{align*}
&\frac{\partial}{\partial u} \mathcal{J}(x, Lu, q) =  \int_0^\infty e\left(- \frac{Nxy}{p^{\ell} a q} \right)   V \left( y\right) \frac{\partial}{\partial u} \left\lbrace  Y_{2 i \nu} + Y_{ - 2 i \nu} \right\rbrace \left( \frac{4 \pi \sqrt{ L u Ny}}{p^{\ell} q }\right)  dy \\
& = \int_0^\infty e\left(- \frac{Nxy}{p^{\ell} a q} \right)   V \left( y\right) \frac{1}{u}  \frac{4 \pi \sqrt{ L u Ny}}{p^{\ell} q }  \left\lbrace  Y_{2 i \nu}^\prime + Y_{ - 2 i \nu}^\prime \right\rbrace \left( \frac{4 \pi \sqrt{ L u Ny}}{p^{\ell} q }\right)  dy \ll 1. 
\end{align*}
This shows that there is no oscillation in function $ \mathcal{J}(x, Ln, q)$ with respect to variable $n$. Also from equation \eqref{bound of mathcal J x q n} we have 

\begin{align} \label{bound for U 1 y}
\int_{\mathbb{R}}   U_1 \left( y\right) e\left( - \frac{ n L y}{ q q^\prime p^{\ell} }\right) dy & =  \int_{\mathbb{R}}  U(y) \mathcal{J}(x, Ly, q) \mathcal{J}(x, Ly, q^\prime) e\left( - \frac{ n L y}{ q q^\prime p^{\ell} }\right) dy \notag \\
& \ll  \frac{p^{\ell} q }{ \sqrt{L N}} \frac{p^{\ell} q^\prime }{ \sqrt{ L N}} \int_1^2  U(y) dy   \ll \frac{p^{ 2 \ell} q  q^\prime}{ L N}, 
\end{align} as $U(y)$ is supported on the interval $[1, 2]$.

 Integrating by parts taking $U_1(y)$ as first function, we observe that the  integral in equation \eqref{T in theorem 2} is negligible if $n \gg p^\varepsilon q q^\prime p^\ell /L.$  Evaluating the above character sum we get the following congruence relation: 
\begin{align*} 
 - p^r\overline{p^{2\ell}}  \overline{m} p^{\ell} q^\prime + p^r\overline{p^{2\ell}}  \overline{m^\prime} p^{\ell} q  - \overline{a + bq} \  \overline{q} q q^\prime + \overline{a + b^\prime q^\prime} \  \overline{q^\prime} q q^\prime + n \equiv 0  \left(\textrm{mod} \ q q^\prime p^{\ell} \right). 
\end{align*} We solve the above congruence modulo $p^{\ell}$  and modulo $  q q^\prime$ respectively to obtain 

\begin{align} \label{congruence in theorem 2}
  - \overline{a + bq} \ q^\prime + \overline{a + b^\prime q^\prime} \   q  + n \equiv 0  \left(\textrm{mod} \  p^{\ell} \right) \  \textrm{and} \  - p^r\overline{p^{2\ell}}  \overline{m} p^{\ell} q^\prime + p^r\overline{p^{2\ell}}  \overline{m^\prime} p^{\ell} q  + n \equiv 0  \left(\textrm{mod} \ q q^\prime \right). 
\end{align}

Writing $n= -  p^r\overline{p^{2\ell}}  \overline{m} p^{\ell} q^\prime + p^r\overline{p^{2\ell}}  \overline{m^\prime} p^{\ell} q + j q q^\prime $, we observe that the number of $n $ satisfying above congruence relation is same as the number of $j$'s. Since we also have $n\ll N^\varepsilon q q^\prime$,  we conclude $j\ll N^\varepsilon$.  Hence the number of solutions of $n$ satisfying the above congruence relation modulo $q q^\prime$, and $n\ll q q^\prime N^\varepsilon$ is bounded by $N^\varepsilon$. For congruence relation modulo $p^{\ell} $ in the above equation,  we substitute the change of variable $a + bq = \alpha $ and $ a + b^\prime q^\prime = \alpha^\prime$  to obtain  
\begin{equation} \label{relation alpha }
\overline{\alpha} \ q^\prime + \overline{\alpha^\prime} \   q  + n \equiv 0  \left(\textrm{mod} \  p^{\ell} \right). 
\end{equation} We record the bound for $ \mathcal{T} $ in the following lemma:

\begin{lemma} \label{lemma mathcal T}
Let $ \mathcal{T} $ be as given in equation \eqref{mathcal T}. We have  

\begin{align*}
\mathcal{T} = L \sideset{}{^\dagger}\sum_{ \substack{n\ll p^\varepsilon q q^\prime p^\ell/L  } } \int_{\mathbb{R}}   U_1 \left( y\right) e\left( - \frac{ n L y}{ q q^\prime p^{\ell} }\right) dy, 
\end{align*}
where we $\dagger$ in above summation denote that $n$ satisfies the congruence relation given in equation \eqref{congruence in theorem 2}. 
\end{lemma}

Substituting the bound for $\mathcal{T} $ in equation \eqref{S 2 hat L} we obtain 

\begin{align} \label{D and ND in thm 1} 
& \hat{S_2} (L) = L   \int_{x\ll q/Q}  \int_{x^\prime\ll q^\prime/Q}  \sum_{m \ll \frac{N^\varepsilon Q p^r}{N} }  \sum_{m^\prime \ll \frac{N^\varepsilon Q p^r}{N} }  \sum_{1\leq q \leq Q} \sum_{1\leq  q^\prime \leq Q} \frac{\chi(q)}{a q^2} \frac{\overline{\chi} (q^\prime)}{a^\prime q^{\prime 2} }  \sideset{}{^\dagger}\sum_{ \substack{n\ll p^\varepsilon q q^\prime p^\ell/L  } }  \sideset{}{^\star}\sum_{\alpha( p^{\ell})} \sideset{}{^\star}\sum_{ \alpha^\prime ( p^{\ell})}   \notag \\  
  & \hspace{10pt} \times   \overline{\chi} \left( m- \alpha p^{r- \ell}  \right) \chi \left( m^\prime-  \alpha^\prime  p^{r- \ell}  \right) \mathfrak{I} (x , q, m)  \mathfrak{I} (x , q^\prime, m^\prime) \int_{\mathbb{R}}   U_1 \left( y\right) e\left( - \frac{ n L y}{ q q^\prime p^{\ell} }\right) dy \ dx \ dx^\prime  \notag \\
   & \hspace{2cm}= \hat{S_2} (D) + \hat{S_2} (ND),
\end{align} where $\alpha$ and $\alpha^\prime$ are related by the congruence relation given in equation \eqref{relation alpha }, and $ \hat{S_2} (D)$ (respectively $ \hat{S_2} (ND)$) is contribution of the diagonal terms (respectively the off-diagonal terms). 
The contribution of the diagonal terms ($\alpha = \alpha^\prime, m= m^\prime $ and $q=q^\prime$) is bounded by (using $ \mathfrak{I} (x , q, m) \ll 1$, bound from the equation \eqref{bound for U 1 y} and sum over $n$ satisfying the congruence relation given in equation \eqref{congruence in theorem 2} is bounded by $p^\epsilon p^\ell /N$):

\begin{align} \label{diagonal in theorem 1}
\hat{S_2} (D) & \ll L  \int_{x\ll q/Q} \int_{x^\prime\ll q^\prime/Q}  \sum_{m \ll \frac{ Q p^r}{N} } \sum_{1\leq q \leq Q} \frac{1}{q^4} \sideset{}{^\dagger}\sum_{ \substack{n\ll p^\varepsilon q q^\prime p^\ell/L  } }   \sum_{\alpha( p^{\ell})}  \frac{\left| \mathfrak{I} (x , q, m)\right|^2}{a^2} \int_{\mathbb{R}} |U_1 (y)| dy \ dx \ dx^\prime \notag \\
& \ll L  N^\varepsilon    \int_{x\ll q/Q} \frac{ Q p^r}{N}  \sum_{1\leq q \leq Q} \sideset{}{^\dagger}\sum_{ \substack{n\ll p^\varepsilon q q^\prime p^\ell/N  } }  \frac{1}{ a^2 q^4} \   p^{  \ell} \  \frac{p^{ 2 \ell} q^2}{ L N}  dx  \notag \\
& \ll   \frac{p^{ 3 \ell} N^\varepsilon }{   N}  \frac{ Q p^r}{N} \frac{p^\ell}{L}  \sum_{1\leq q \leq Q} \frac{1}{ a^2 q^2} \times \frac{q}{Q}  \ll \frac{p^{ 3 \ell} N^\varepsilon }{ Q^2  N}   \frac{  p^r}{N}  \frac{p^\ell}{L},  
\end{align}
as $ a\asymp Q.$ Substituting the value of $\alpha^\prime$ from the congruence relation given in equation \eqref{relation alpha }, we see that the contribution of the off-diagonal term is  given by:

\begin{align} \label{non diagonal in theorem 1} 
\hat{S_2} (ND) & =  L  \int_{x\ll q/Q}  \int_{x^\prime \ll q^\prime /Q} \sum_{m \ll Q } \sum_{ m^\prime \ll Q } \sum_{1\leq q \leq Q} \sum_{1\leq  q^\prime \leq Q} \frac{\chi(q)}{a q^2} \frac{\overline{\chi} (q^\prime)}{ a^\prime q^{\prime 2} }  \sideset{}{^\dagger}\sum_{ \substack{n\ll p^\varepsilon q q^\prime p^\ell/N  } }  \notag \\ 
  & \hspace{20pt} \times  \left\lbrace \sideset{}{^\star}\sum_{\alpha( p^{\ell})}  \overline{\chi} \left( m- \alpha p^{r- \ell}  \right) \chi \left( m^\prime +   \alpha q  \ \overline{n+ q^\prime} p^{r- \ell}  \right) \right\rbrace \mathfrak{I} (x , q, m)  \mathfrak{I} (x , q^\prime, m^\prime) \notag \\
 & \hspace{20pt} \times \int_{\mathbb{R}}   U_1 \left( y\right) e\left( - \frac{ n L y}{ q q^\prime p^{\ell} }\right) dy  \ dx \ dx^\prime.  
\end{align} 
Next we evaluate the exponential sum in the  above equation. 


\subsection{Evaluation of the character sum} 
 In this subsection we shall prove the following lemma

\begin{lemma} \label{character sum in thm 1} 
Let $\mathcal{A}$ be the character sum given by

\begin{align*}
\mathcal{A}:= \sideset{}{^\star}\sum_{\alpha( p^{\ell})}  \overline{\chi} \left( m- \alpha p^{r- \ell} \right)  \chi \left( m^\prime +   \alpha q  \ \overline{n+ q^\prime} p^{r- \ell}  \right), 
 \end{align*} with $ \ell = 2 \floor{\frac{r}{3}} $. We have 
 \[
 \mathcal{A} \ll p^{\ell/2 + \epsilon}. 
 \]
\end{lemma}
 \begin{proof}

 Applying the change of variable $\alpha = \alpha_1 p^{\ell/ 2} + \alpha_2 $ where $\alpha_1$ and $ \alpha_2$ run over residue classes modulo $ p^{\ell/ 2}$, the above character sum reduces to 

\begin{align*}  
\mathcal{A} &=  \sideset{}{^\star} \sum_{\alpha_2( p^{\ell/ 2})} \sideset{}{^\star} \sum_{\alpha_1( p^{\ell/ 2})} \overline{\chi} \left( m- \alpha_2 p^{r-\ell}  -  \alpha_1 p^{2(r-\ell)}   \right)  \chi \left( m^\prime +   \left( \alpha_1 p^{r-\ell} + \alpha_2 \right) q  \ \overline{n+ q^\prime} p^{r-\ell}  \right) \notag \\
& =  \sideset{}{^\star} \sum_{\alpha_2( p^{\ell/2})}  \sideset{}{^\star} \sum_{\alpha_1( p^{\ell/2})}   \chi  \left\lbrace \left( m^\prime +  q  \ \overline{n+ q^\prime} p^{r-\ell}   \alpha_2  +  q  \ \overline{n+ q^\prime} \alpha_1 p^{2(r-\ell)}  \right)  \right. \notag \\ 
 & \left. \hspace{3cm} \left( \overline{m - \alpha_2p^{r-\ell} } +  (\overline{m - \alpha_2p^{r-\ell} })^2  \alpha_1 \ p^{ 2(r-\ell)}\right)  \right\rbrace ,   \notag 
 \end{align*} 
 as $ \overline{ m- \alpha_2 p^{r-\ell}  -  \alpha_1 p^{2(r-\ell)}   } =  \overline{m - \alpha_2p^{r-\ell} } +  (\overline{m - \alpha_2p^{r-\ell} })^2  \alpha_1 \ p^{ 2(r-\ell)} ( \textrm{mod} \ p^r )$. Which reduces to 
 \begin{align*}
\mathcal{A} &=\sideset{}{^\star} \sum_{\alpha_2( p^{\ell/2})}  \sideset{}{^\star} \sum_{\alpha_1( p^{\ell/2})}   \chi  \left( A(\alpha_2)  + B(\alpha_2) \alpha_1  p^{ 2(r-\ell)} \right) \notag\\ 
&= \sideset{}{^\star} \sum_{\alpha_2( p^{\ell/2})}  A(\alpha_2) \sideset{}{^\star} \sum_{\alpha_1( p^{\ell/2})}   \chi  \left( 1+  \overline{A(\alpha_2)} B(\alpha_2) \alpha_1  p^{ 2(r-\ell)} \right),  
 \end{align*} 
 
 where $ A(\alpha_2) = m^\prime \  \overline{m - \alpha_2p^{r-\ell} } + q \ \overline{n+ q^\prime} \alpha_2  p^{ r-\ell}  \overline{m - \alpha_2p^{r-\ell} }  $  and $ B(\alpha_2) = m^\prime \  ( \overline{m - \alpha_2p^{r-\ell} })^2  + q \ \overline{n+ q^\prime}  \   \overline{m - \alpha_2p^{r-\ell} }  .$ Note that $ ( A(\alpha_2), p ) = 1$, otherwise $ \chi  \left( A(\alpha_2)  + B(\alpha_2) \alpha_1  p^{ 2r/3} \right) = 0. $ For a fixed $\alpha_2$,   $\chi  \left( 1+  \overline{A(\alpha_2)} B(\alpha_2) \alpha_1  p^{ 2(r-\ell)} \right) : = \chi  \left( 1+  C(\alpha_2) \alpha_1  p^{ 2(r-\ell)} \right)$ is  an additive character of modulus $ p^{\ell/2}$, as we have
 
\begin{align*}
\chi  \left( 1+  C(\alpha_2) \alpha_1  p^{ 2(r-\ell)} \right) \chi  \left( 1+  C(\alpha_2) \alpha_1^\prime  p^{ 2(r-\ell)} \right) = \chi  \left( 1+  C(\alpha_2) (  \alpha_1 + \alpha_1^\prime)  p^{ 2(r-\ell)} \right), 
\end{align*} as we have $4(r-\ell) \geq r$. Hence there exists an integer $ b$ (uniquely determined modulo $ p^{\ell/2}$) such that 
\[
\chi  \left( 1+  C(\alpha_2) \alpha_1  p^{ 2(r-\ell)} \right) =  e \left( \frac{ \alpha_1 \  bC(\alpha_2) }{p^{\ell}}\right).
\]
Executing the sum over $\alpha_1$ given in equation \eqref{character sum in thm 1} we have 

\begin{align}
\mathcal{A} &= p^{\ell/2}  \sideset{}{^\star} \sum_{\substack {\alpha_2( p^{\ell/2}) \\ b C(\alpha_2) \equiv 0 (\textrm{mod} \ p^{\ell/2}) }} \chi \left( A(\alpha_2)\right) \ll   p^{\ell/2 + \varepsilon}. 
\end{align}

\end{proof}

Substituting the bound for the  character sum in equation \eqref{non diagonal in theorem 1} and using the bounds of $U_1(y)$ given in equation \eqref{bound for U 1 y}, we have

\begin{align} 
\hat{S_2} (ND)  & \ll p^\varepsilon L \sum_{m \ll \frac{ Q p^r}{N} } \sum_{m^\prime \ll \frac{ Q p^r}{N} }   \sum_{1\leq q \leq Q} \sum_{1\leq  q^\prime \leq Q} \ \sideset{}{^\dagger} \sum_{ \substack{n\ll p^\varepsilon q q^\prime p^\ell/L  } }  \frac{1 }{a q^2} \frac{1 }{ a^\prime q^{\prime 2} }  \ p^{\ell/ 2} \  \frac{p^{ 2 \ell} q  q^\prime}{ L N} \notag \\ 
 & \ll p^\varepsilon  \frac{p^{ 5\ell/2 } }{ Q^2 N} \left( \frac{ Q p^r}{N}\right)^2  \frac{p^\ell}{L}\ll  p^\varepsilon  \frac{p^{ 3\ell/2 } }{ Q^2} \left( \frac{  p^r}{N}\right)^2 \frac{p^\ell}{L},
\end{align}   as $a, a^\prime \asymp Q$,  $Q^2 = N/ p^{ \ell }$ and dagger on summation over $n$ shows that $n$ satisfies the congruence relation modulo $q q^\prime$ as given in equation  \eqref{congruence in theorem 2}.  Substituting the bounds for $ \hat{S_2} (D)$ and $ \hat{S_2} (ND)$ in equation \eqref{D and ND in thm 1} we have 
\begin{align*}
\hat{S_2} (L) \ll  p^\varepsilon  \frac{p^\ell}{L} \left(  \frac{p^{ 3 \ell}  }{ Q^2 N} \frac{  p^r}{N}  +    \frac{p^{ 3\ell/2 } }{ Q^2}  \frac{  p^{2 r}}{N^2}\right) \ll p^\varepsilon \frac{p^\ell}{L Q^2 N^2} p^{ \frac{3 \ell}{2 }} p^r \left(  p^{ \frac{3 \ell}{2 }} + p^r \right) . 
\end{align*}

Substituting the bound for $ \hat{S_2} (L) $ in equation \eqref{define S 2 hat L} we obtain

\begin{align} 
 S_2^{ +} (N) & \ll   p^\varepsilon  \frac{N^2 \left| \tau_{\chi}  \right|}{ p^{ r + 2\ell}}    \sum_{\substack{ L \ll P_1 \\ L \textrm{-} \textrm{dyadic}}}  L^{1/2} \times  \frac{p^{ \frac{ \ell}{2 } }}{\sqrt{L} Q N} p^{ \frac{3 \ell}{4 } }p^{\frac{r}{2} } \left(  p^{ \frac{3 \ell}{4 } } + p^{\frac{r}{2}} \right)  \notag\\
 & \ll p^\varepsilon \frac{N}{Q  p^{ \frac{3 \ell}{4 } } } \left(  p^{ \frac{3 \ell}{4 } } + p^{\frac{r}{2}} \right)  \ll p^\varepsilon  N^{1/2 } \left(  p^{ \frac{ \ell}{2 } } + \frac{p^{\frac{r}{2}}}{ p^{ \frac{ \ell}{4 } }} \right)   \ll p^\varepsilon  N^{1/2 } p^{\frac{1}{2} \left( r - \floor{r/3}\right) + \varepsilon}  ,
 \end{align} 
 
 as $ \ell = 2 \floor{\frac{r}{3}} $ and  $Q= N^{1/2}/  p^{  \ell/2}$. This proves Proposition \ref{proposition 1}.

 \section{Proof of Theorem \ref{theorem 2} }

 In order  to prove Theorem \ref{theorem 2}, we shall first establish the following bound. 

\begin{proposition}
We have 
\begin{align*}
S(N) & \ll \begin{cases}
  N^{1+\varepsilon}    \ \  \ \ \  \ \ \ \ \ \ \   \textrm{if }  \ \ \ \   1\leq N \ll P^{3/4+ \varepsilon} \\ 
   P^\varepsilon N^{\frac{1}{2} } \left( P^{1/4} P_1^{1/4} + P^{1/2} P_1^{ - 1/4} \right) \  \ \ \ \     \textrm{if }  \ \ \ \  P^{3/4+ \varepsilon}   \ll N \ll P^{1+\varepsilon}, 
\end{cases}
\end{align*} where $P = P_1 P_2$.

\end{proposition}

\subsection{Application of circle method} 

As in Theorem \ref{theorem 1}, we shall analyse the sum $S(N)$ for the  case  of the Maass forms (holomorphic case is similar, even simpler). We first separate the oscillation of the Fourier coefficients $\lambda_f(n)$ and $\chi(n)$ using the delta symbol. We write $S(N)$ as 

\begin{align*}
S(N)=   \mathop{\sum \sum}_{ {m, n=1} }^\infty \lambda_f (n)  \chi_1(n) \chi_2(m) \delta \left( n -m \right) V\left( \frac{n}{N} \right) V_1\left( \frac{m}{N} \right),  
\end{align*}
 where $V_1(y)$ is another smooth function, supported on the interval $[1/2, 3]$, $V_1(y) \equiv 1$ for $y \in [1,2]$ and satisfies $y^j V^{(j)} \ll _j 1.$ To analyse the sum $S(N)$ we  shall use the conductor lowering mechanism,  as used in Theorem 1.  The integral equation $n = m$ is equivalent to the congruence $ n \equiv m (\textrm{mod} \ P_1)$ and the integral equation $(n-m)/P_1 = 0$. We now have
 \[
 S(N) := \mathop{\sum \sum}_{ \substack {{m, n=1} \\  P_1 \mid (n-m) }}^\infty \lambda_f (n)  \chi_1(n) \chi_2(m) \delta \left( \frac{n -m}{P_1} \right) V\left( \frac{n}{N} \right) V_1\left( \frac{m}{N} \right). 
 \]  Now using Lemma \ref{circlemethod}  for the expression of $\delta(n)$, we have 
 
\[
S(N)= S_3^+(N) + S_3^-(N),
\] with
\begin{align*} 
S_3^{ \pm} (N) &=   \int_0^1  
\mathop{\sum  \sideset{}{^\star}\sum}_{1\leq q\leq Q < q \leq q+Q}  \frac{1}{aq} \mathop{\sum \sum}_{ \substack {{m, n=1} \\  P_1 \mid (n-m) }}^\infty  \lambda_f(n) \chi_1(n) \chi_2(m) \notag\\
 & \hspace{1cm} \times e\left( \pm \frac{ \overline{a}(n-m)/P_1}{q} \mp \frac{x (n-m)/P_1 }{ aq}\right) V\left( \frac{n}{N} \right) V_1\left( \frac{m}{N} \right) dx. 
\end{align*}  We choose $Q=(N/P_1)^{1/2}$. We detect the congruence relation $n \equiv m (\textrm{mod} \ P_1)$ in the above expression using the exponential sum to get

\begin{align*}
S_3^{ \pm} (N) &=    \int_0^1  
\mathop{\sum  \sideset{}{^\star}\sum}_{1\leq q\leq Q < q \leq q+Q}  \frac{1}{aq P_1}  \sum_{b ( P_1)}\mathop{\sum \sum}_{ \substack {{m, n=1} }}^\infty  \lambda_f(n) \chi_1(n) \chi_2(m) \notag\\
 & \hspace{1cm} \times e\left( \pm \frac{ (\overline{a}+ bq) (n-m)}{P_1 q} \right) e\left( \mp \frac{ x (n-m)}{ a P_1 q} \right) V\left( \frac{n}{N} \right) V_1\left( \frac{m}{N} \right) dx. 
\end{align*} In the rest of the paper we will analyse the sum $S_3^{+}(N)$ ( analysis of $S_3^{-}(N)$ is just similar). We have

 \begin{align} \label{splus} 
S_3^{ +} (N) &=   \int_0^1  
\mathop{\sum  \sideset{}{^\star}\sum}_{1\leq q\leq Q < q \leq q+Q}  \frac{1}{aq P_1}  \sum_{b ( P_1)} \left\lbrace\sum_{n=1}^\infty  \lambda_f(n) \chi_1(n)  e\left(  \frac{ (\overline{a}+ bq) n}{P_1 q} \right) e\left(  \frac{ -x n}{  P_1 a q} \right)  V\left( \frac{n}{N} \right) \right\rbrace \notag\\  
 & \hspace{1cm}    \left\lbrace\sum_{m=1}^\infty   \chi_2(m)  e\left( - \frac{ (\overline{a}+ bq) m}{P_1 q} \right) e\left(  \frac{ m x }{  P_1 a q} \right)  V_1\left( \frac{m}{N} \right) \right\rbrace dx.  
\end{align}

\subsection{Appying Poisson summation Formula}
We shall now apply the Poisson summation formula to sum over $m$ in equation  \eqref{splus} as follows. Writing $m=\beta + l P_1 P_2 q$ and then applying the Poisson summation formula to the sum over $l $ we have 

\begin{align*}
& \sum_{m=1}^\infty   \chi_2(m)  e\left( - \frac{ (\overline{a}+ bq) m}{P_1 q} \right) e\left(  \frac{ m x }{  P_1 a q} \right)  V_1\left( \frac{m}{N} \right) \\
 &  = \sum_{\beta(P_1 P_2 q)} \chi_2(\beta) e\left( - \frac{ (\overline{a}+ bq) \beta}{P_1 q} \right) \sum_{ l \in \mathbb{Z}}  V_1\left( \frac{ \beta + l P_1 P_2 q }{N} \right)   e\left(  \frac{ (\beta + l P_1 P_2 q) x }{  P_1 a q} \right) \\
  &  = \sum_{\beta(P_1 P_2 q)} \chi_2(\beta) e\left( - \frac{ (\overline{a}+ bq) \beta}{P_1 q} \right) \sum_{ m \in \mathbb{Z}} \int_{\mathbb{R}} V_1\left( \frac{ \beta + y P_1 P_2 q }{N} \right)   e\left(  \frac{ (\beta + y P_1 P_2 q) x }{  P_1 a q} \right) e(-my) dy. 
\end{align*} We now apply the change of variable $ ( \beta + y P_1 P_2 q)/N = z $ to obtain

\begin{align} \label{first poisson}
&  \frac{N}{P_1 P_2 q}\sum_{ m \in \mathbb{Z}} \sum_{\beta(P_1 P_2 q)} \chi_2(\beta) e\left( - \frac{ (\overline{a}+ bq) \beta}{P_1 q}  + \frac{m \beta}{ P_1 P_2 q}\right) \int_{\mathbb{R}} V_1 (y)  e\left(  \frac{  N x y }{  P_1 a q} \right) e\left(  \frac{ - N m y }{  P_1 P_2 q} \right) dy  \notag \\
&:= \frac{N}{P_1 P_2 q} \sum_{ m \in \mathbb{Z}} \mathfrak{C} (b , q) \mathfrak{I} (x , q, m), 
\end{align} 
where $ \mathfrak{C} (b , q)$ is a character sum and $\mathfrak{I} (x , q, m) $ is the integral in the above expression. Writing $\beta = a_2 P_1 q \overline{P_1 q} + a_1 P_2 \overline{P_2}$, where $ a_2 (\textrm{mod} \  P_2)$, $a_1 (\textrm{mod} \ P_1 q) $  and $ P_1 q \overline{P_1 q} \equiv 1 (\textrm{mod} \ P_2)$, the above character sum can be written as 
\begin{align*}
& \sum_{a_1 (P_1 q)} e \left(a_1 \frac{ ((\overline{a} + bq )P_2 +m ) \overline{P_2}}{ P_1 q}\right) \sum_{a_2 (P_2)} \chi_2(a_2) e \left(a_2 \frac{ ((\overline{a} + bq )P_2 +m ) \overline{P_1 q}}{ P_2} \right) \\
&= \sum_{a_1 (P_1 q)} e \left( a_1\frac{ (\overline{a} + bq  +m  \overline{P_2}  )}{ P_1 q}\right) \sum_{a_2 (P_2)} \chi_2(a_2) e \left( a_2 \frac{  m  \overline{P_1 q}}{ P_2} \right) \\
& = \begin{cases}
  P_1 q \ \  \tau_{\chi_2} \ \   \chi_2 (\overline{m} P_1 q)  \ \  \ \ \   \textrm{if }  \ \ \ \ \ \overline{a} + bq  +m  \overline{P_2} \equiv 0 ( \textrm{mod} \ P_1 q ), \ \ \   P_2 \nmid m \\ 
0   \ \ \ \     \textrm{otherwise} .
\end{cases}
\end{align*}
From the congruence relation, $a (\textrm{mod} \ q)$  and $b (\textrm{mod} \ P_1)$ can be determined.  Next we consider the integral in equation \eqref{first poisson}. Integrating by parts $j$-times and using $ V_1^{(j)} (y) \ll 1$ we have
\begin{align*}
\mathfrak{I} (x , q, m) \ll \left(1 + \frac{  N x  }{  P_1 a q} \right)^{ j} \left(  \frac{P_1 P_2 q}{ Nm}\right)^j. 
\end{align*} 
We observe that $ \mathfrak{I} (x , q, m)$ is negligibly small if $ m \gg (\sqrt{P_1} P_2 P^\varepsilon)/ \sqrt{N}$, i.e. if $m\gg P^\varepsilon P_2 / Q $.  We record this result in the following lemma

\begin{lemma} We have 
 
 \begin{align}  
S_3^{ +} (N) &=  \int_0^1  
\sum_{1\leq q\leq Q}  \frac{1}{q P_1}  \left\lbrace\sum_{n=1}^\infty  \lambda_f(n) \chi_1(n)  e\left(  \frac{ (\overline{a}+ bq) n}{P_1 q} \right) e\left(  \frac{- x n}{  P_1 a q} \right) V\left( \frac{n}{N} \right)  \right\rbrace  \notag\\  
 & \hspace{1cm}  \times \left\lbrace \frac{ \tau_{\chi_2} \ N  \chi_2 ( P_1 q)  }{ P_2 }  \sum_{m \ll \frac{P_2}{Q} P^{\varepsilon} }    \    \chi_2 (\overline{m} )   \frac{\mathfrak{I} (x , q, m)}{a} \right\rbrace dx + O_A \left( P^{-A}\right),  
\end{align}  
 for any positive constant $A >0$. 
 
\end{lemma}

\subsection{Applying Voronoi summation formula}
 Substituting the following identity
\[
\chi_1(n) = \frac{1}{\tau ( \overline{\chi_1} ) } \sum_{ \alpha (P_1)} \overline{ \chi_1} (\alpha) e\left(\frac{\alpha n}{P_1} \right), 
\] we have

\begin{align} \label{s 4 n}
 & \sum_{n=1}^\infty  \lambda_f(n) \chi_1(n)  e\left( \frac{ (\overline{a}+ bq) n}{P_1 q} \right) e\left(  \frac{ -x n}{  P_1 a q} \right)  V\left( \frac{n}{N} \right)  =  \frac{1}{\tau ( \overline{\chi_1} ) } \sum_{ \alpha (P_1)} \overline{ \chi_1} (\alpha)  \sum_{n=1}^\infty  \lambda_f(n)  \notag\\
 &  \hspace{2cm} \times e\left(  \frac{ (\overline{a}+ bq + \alpha q) n}{P_1 q} \right) e\left(  \frac{ -x n}{  P_1 a q} \right)  V\left( \frac{n}{N} \right)  :=  \frac{1}{\tau ( \overline{\chi_1} ) } \sum_{ \alpha (P_1)} \overline{ \chi_1} (\alpha)  S_4 (N).
\end{align} 

 We have $ (\overline{a}+ bq + \alpha q, q) = 1$. Given $a, $ and $b$ there exists at-most one $\alpha$ such that $ \overline{a}+ bq + \alpha q \equiv 0 (\textrm{mod} \  P_1 )$.  For the rest of $\alpha$ we apply the  Voronoi summation formula to the sum over $n$(In the case where $ \overline{a}+ bq + \alpha q \equiv 0 (\textrm{mod} \  P_1 )$, we first take out the power of $P_1$ from $ \overline{a}+ bq + \alpha q $, and then proceed as bellow. In this case, since conductor is smaller that $P_1 q$, the Voronoi summation formula gives us more saving. Also in this case $\alpha$ is uniquely determined, which gives us  saving over  $\alpha$ summation). We substitute $g(n)=  e(-nx/ P_1aq) V(n/N)$ in Lemma \ref{voronoi Maass} to get

\begin{align*}
S_4^{\pm} (N) :=   \frac{1}{P_1 q}    \sum_{\pm} \sum_{n\geq 1} \lambda_f(\mp n) e\left( \pm  n \frac{ \overline{\overline{a}+ bq + \alpha q } }{P_1 q} \right) H^{\pm} \left( \frac{ n}{q^2}, \frac{ N x}{aq}\right), 
\end{align*} where  $ H^{\pm}$  are given in Lemma \ref{voronoi Maass}.  We shall estimate  $S_4 ^{-}(N) $ ( $ S_4 ^{+}(N)$ is similar). We have
\begin{align*}
  H^{-} \left( \frac{ n}{q^2}, \frac{ N x}{P_1aq}\right) =  \int_0^\infty e\left(- \frac{xy}{P_1 a q} \right)  V \left( \frac{y}{N}\right) \left\lbrace  Y_{2 i \nu} + Y_{ - 2 i \nu}\right\rbrace \left( \frac{4 \pi \sqrt{ny}}{P_1 q }\right)  dy.
\end{align*}
 Applying the change of variable $y/ N= z$, we have

\begin{align} \label{I x n q} 
 H^{-} \left( \frac{ n}{q^2}, \frac{ N x}{P_1aq}\right) =   N\int_0^\infty e\left(- \frac{Nxy}{P_1 a q} \right)   V \left( y\right) \left\lbrace  Y_{2 i \nu} + Y_{ - 2 i \nu}\right\rbrace \left( \frac{4 \pi \sqrt{n Ny}}{P_1 q }\right)  dy : = N \ I (x, n, q), 
\end{align}  
where $ I (x, n, q)$ denotes the integral in above equation. Substituting the asymptotic  for  $ U_{\pm2 i\nu}(x)$ given in equations \eqref{bessel} and \eqref{parti der of bessel 2}, we have the first term of integral in equation \eqref{I x n q} is equal to (estimation of second term is similar)
 
\begin{align} \label{I plus minus}
I^{\pm} (x, n) := \int_0^\infty e^{ i \left( - \frac{ 2 \pi N x y}{P_1 a q } \pm i \frac{4 \pi \sqrt{n Ny}}{P_1 q }  \right) } V(y)  U_{2 i\nu}^\pm \left( \frac{4 \pi \sqrt{n Ny}}{P_1 q } \right) dy, 
\end{align} 
where $U^+(y) = U(y)$ and $U^-(y) = \overline{U}(y)$. We observe that the integral $I^{-} (x, n)$ has no stationary point. By equation \eqref{unstationary},  $I^{-} (x, n)$ is negligibly small. For $I^{+} (x, n)$ we apply second statement of Lemma \ref{exponential inte} with 
\[
f(y) = - \frac{ 2 \pi N x y}{P_1 a q } + i \frac{4 \pi \sqrt{n Ny}}{P_1 q } \  \ \ \textrm{and} \ \ \ \ \  g(y) =  V(y)  U_{2 i\nu}^\pm \left( \frac{4 \pi \sqrt{n Ny}}{P_1 q } \right). 
\]
We have 
\begin{align*}
f^\prime(y)= - \frac{ 2 \pi N x }{P_1 a q }  + \frac{2 \pi \sqrt{n N}}{ \sqrt{y} P_1 q }, \ \ \ \  f^{\prime \prime}(y) = - \frac{ \pi \sqrt{n N}}{ y^{3/2} P_1 q }.
\end{align*} We observe that if we denote the stationary point by $ y_0$ then we have
\[
|f^{\prime \prime}(y_0)| \asymp  \frac{  \sqrt{n N}}{  P_1 q } \ \ \ \ \textrm{as} \ \ y_0 \asymp 1.
\] 
Using  $ U_{\pm2 i\nu} (x) \ll_{  \nu} (1+x)^{-1/2}$, and applying the second statement of the Lemma \ref{exponential inte}, we have

\begin{align} \label{bound I x n q}
I(x, n, q) \ll \frac{P_1 q }{ \sqrt{n N}}, 
\end{align}
where $ I(x, n, q) $ is given in equation \eqref{I x n q}.  Also, integrating by parts we have
 
 \[
 I^{\pm} \ll_j \left( \frac{Nx}{P_1 a q} + 1\right)^j \left( \frac{P_1 q}{\sqrt{nN}} \right)^j. 
 \] 
 Hence integral is negligibly small if 
\[
\frac{P_1 q }{ \sqrt{n N}} \ll P^{-\epsilon} \Rightarrow n \gg P^\varepsilon P_1, 
\] 
as $q\leq Q = N^{1/2}/P_1^{1/2}.$ We record this result in the following lemma. After the application of the Poisson and the Voronoi summation formula we have the following expression for $S_3^+(N)$. 

\begin{lemma} We have

\begin{align*} 
S_3^+(N) & =  \int_0^1  \sum_{1\leq q\leq Q}  \frac{1}{ q P_1} \frac{N}{P_2} \tau(\chi_2)\chi_2 (P_1 q)  \sum_{m \ll \frac{P_2}{Q} P^{\varepsilon} }    \    \chi_2 (\overline{m} )   \frac{\mathfrak{I} (x , q, m)}{a}  \frac{1}{ \tau ( \overline{\chi_1} )} \sideset{}{^\star}\sum_{ \alpha (P_1)} \overline{ \chi_1} (\alpha) \\
& \times \frac{N}{P_1 q}  \sum_{n\ll P_1} \lambda_f(n) e\left( -  n \frac{ \overline{ m\overline{P_2}+  \alpha q } }{P_1 q} \right)  I(x, n, q)  + O_A\left( P^{-A}\right),
\end{align*}
where $A>0$ is any positive constant. 
\end{lemma}
A trivial estimate gives us ( assuming cancellation in character sum $\alpha$) 
\begin{align*}
S_3^+(N) \ll \frac{1}{P_1 a} \frac{N}{P_2} \sqrt{P_2} Q \frac{1}{\sqrt{P_1}} \sqrt{P_1}  \frac{N}{P_1 q}  P_1  \ll N^{1+\varepsilon}. 
\end{align*} 
So we are on the boundary. To get additional  cancellation, we shall now apply the Cauchy inequality and the Poisson summation formula.  Interchanging the order of summation and using the bound for $I(x, n, q)$ given in equation \eqref{bound I x n q}, we have 
 
\begin{align*}
 S_3^+(N) & =  \frac{N}{  P_1 P_2} \chi_2 (P_1 ) \tau(\chi_2) \frac{1}{ \tau ( \overline{\chi_1} )}  \frac{N}{P_1 }   \int_0^1  \sum_{1\leq q\leq Q}  \frac{1}{q^2}  \sum_{n\ll P_1} \lambda_f(n)    I(x, n, q)   \\
& \times   \sum_{m \ll \frac{P_2}{Q} P^{\varepsilon} }    \    \chi_2 (\overline{m} )   \frac{\mathfrak{I} (x , q, m)}{a}    \sideset{}{^*} \sum_{ \alpha (P_1)} \overline{ \chi_1} (\alpha) e\left( -  n \frac{ \overline{ m\overline{P_2}+  \alpha q } }{P_1 q} \right) dx + O_A\left( P^{-A}\right) \\ 
& \ll  C  \times  \int_0^1  \sum_{1\leq q\leq Q}  \frac{1}{q^2}  \sum_{n\ll P_1} \left|\lambda_f(n) \right| \left| I(x, n, q)\right| \\
&   \hspace{1cm} \times \left|\sum_{m \ll \frac{P_2}{Q} P^{\varepsilon} }  \chi_2 ( \overline{m})  \frac{\mathfrak{I} (x , q, m)}{a}    \sideset{}{^*} \sum_{ \alpha (P_1)} \overline{ \chi_1} (\alpha) e\left( -  n \frac{ \overline{ m\overline{P_2}+  \alpha q } }{P_1 q} \right)\right| \\
& \ll C  \frac{P_1}{\sqrt{N}} \sum_{1\leq q\leq Q}  \frac{1}{q}  \sum_{n\ll P_1} \frac{\left|\lambda_f(n) \right|}{\sqrt{n}}  |S_0(q, n)|,  
\end{align*}
 where $C= \frac{N}{  P_1 P_2} | \tau(\chi_2)| \frac{1}{| \tau ( \overline{\chi_1} )|}  \frac{N}{P_1 } ,$  and $ S_0(q, n)$ is the quantity inside the last modulus. We write $ C \frac{P_1}{\sqrt{N}} : = C_0$ Applying the Cauchy inequality  to the sum over $n$ we have

\begin{align} \label{S plus before second cauchy}
S_3^+(N)&\ll C_0  \int_0^1  \sum_{1\leq q\leq Q}  \frac{1}{q} \sum_{\substack{ L \ll P_1 \\ L \textrm{-} \textrm{dyadic}}} \left\lbrace  \sum_{n \in \mathbb{Z}}  \left| S_0(q, n)\right|^2  U\left( \frac{n}{L}\right) \right\rbrace^{1/2} \notag \\ 
 & :=  C_0  \int_0^1  \sum_{1\leq q\leq Q}  \frac{1}{q} \sum_{\substack{ L \ll P_1 \\ L \textrm{-} \textrm{dyadic}}}  \left\lbrace \hat{S}_0 (q) \right\rbrace^{1/2}. 
\end{align} 
Opening the absolute square  of $ S_0(q, n) $ and interchanging the order of summation, we get

\begin{align} \label{s hat q} 
 &\hat{S}_0 (q) = \sum_{m \ll \frac{P_2}{Q} P^{\varepsilon} } \sum_{m^\prime \ll \frac{P_2}{Q} P^{\varepsilon} }  \chi_2 ( \overline{m}) \overline{\chi_2 ( \overline{m^\prime})}  \frac{\mathfrak{I} (x , q, m)}{a}   \frac{\mathfrak{I} (x , q, m^\prime)}{ a^\prime}   \times \sideset{}{^*} \sum_{ \alpha (P_1)}\sideset{}{^*} \sum_{ \alpha^\prime (P_1)}   \chi_1 (\alpha) \overline{\chi_1 (\alpha^\prime)}  \ \ \  T, 
\end{align} where

\begin{align*}
T = \sum_{n\in \mathbb{Z}} e\left( -  n \frac{ \overline{ m\overline{P_2}+  \alpha q } }{P_1 q}  +  n \frac{  \overline{m^\prime \overline{P_2}+  \alpha^\prime q}  }{P_1 q} \right)  U\left( \frac{n}{L}\right). 
\end{align*}

\subsection{ Second application of Poisson summation formula}
Writing $n= \beta + l P_1 q$, $l \in \mathbb{Z}$ we have 

\begin{align*}
T &= \sum_{\beta (P_1 q)}  e\left( -  \beta  \frac{ \overline{ m\overline{P_2}+  \alpha q } }{P_1 q}  +  \beta \frac{ \overline{ m^\prime \overline{P_2}+  \alpha^\prime q}  }{P_1 q} \right) \sum_{l \in \mathbb{Z}}  U\left( \frac{\beta + l P_1 q}{L}\right) \\
&= \frac{L}{P_1 q} \sum_{n\in \mathbb{Z}} \sum_{\beta (P_1 q)}  e\left( -  \beta  \frac{ \overline{ m\overline{P_2}+  \alpha q } }{P_1 q}  +  \beta \frac{\overline{  m^\prime \overline{P_2}+  \alpha^\prime q } }{P_1 q}  + \frac{n}{P_1 q} \beta \right) \int_{\mathbb{R}} U(y) e \left( - \frac{nL y}{P_1 q}\right) dy. 
\end{align*} Integrating by parts we observe that the  integral in the above expression is negligible small if $n\gg (P_1 q P^\varepsilon)/L$. Evaluating the character sum we  obtain the congruence relation 

\begin{equation} \label{congruence in burgess}
\overline{ m\overline{P_2}+  \alpha q }   +    \overline{m^\prime \overline{P_2}+  \alpha^\prime q }   + n \equiv 0 \ ( \textrm{mod} \ P_1 q). 
\end{equation}
We obtain
\begin{align*}
T = L \sum_{\substack{ n\ll  (P_1 q P^\varepsilon)/L  \\  \overline{ m\overline{P_2}+  \alpha q }   +    \overline{m^\prime \overline{P_2}+  \alpha^\prime q }   + n \equiv 0 \ ( \textrm{mod} \ P_1 q)} }  \int_{\mathbb{R}} U(y) e \left( - \frac{nL y}{P_1 q}\right) dy + O_A \left(P^{-A} \right).
\end{align*}
   The above congruence relation modulo $q$ reduces to $ \overline{ m }   +    \overline{m^\prime  }   + n \overline{P_2} \equiv 0 \ ( \textrm{mod} \ q) \Rightarrow \overline{m^\prime  } \equiv - \overline{ m } -   n \overline{P_2} ( \textrm{mod} \ q).$ We now  solve $\alpha^\prime$ mod $P_1$ in terms of $\alpha$ from the above congruence relation.  Reducing the congruence relation given in equation \eqref{congruence in burgess} modulo $P_1$, we obtain
\begin{align*}
& m^\prime \overline{P_2} + \alpha^\prime q + m \overline{P_2} + \alpha q + n(m^\prime \overline{P_2} + \alpha^\prime q  ) ( m \overline{P_2} + \alpha q ) \equiv 0  \ (\textrm{mod} \ P_1). 
\end{align*} 
We now substitute the change of variable $  m \overline{P_2} + \alpha q = \beta $  to get
\begin{align} \label{relation beta alpha prime in thm 1}
& \alpha^\prime q ( 1+ n \beta) + \beta ( 1+ m^\prime \overline{P_2} n )  + m^\prime  \overline{P_2}  \equiv 0  \ (\textrm{mod} \ P_1) \notag \\ 
& \Rightarrow   \alpha^\prime  \equiv -  \frac{ \beta ( 1+ m^\prime \overline{P_2} n )  + m^\prime  \overline{P_2}}{( 1+ n \beta) q} \ (\textrm{mod} \ P_1). 
\end{align} Now the character summation over $ \alpha, \alpha^\prime$ reduces to 
\begin{align} \label{character sum in thm 2}
 \mathcal{B} := \sum_{\beta (\textrm{mod} \ P_1) }  \chi_1 \left( \frac{ \overline{ q} ( \beta - m \overline{P_2} )  ( 1+ n \beta) q}{ \beta ( 1+ m^\prime \overline{P_2} n )  + m^\prime  \overline{P_2}} \right) : =  \sum_{\beta (\textrm{mod} \ P_1) } \chi_1 (g(\beta)). 
\end{align} Since denominator of $g(\beta)$ is a linear polynomial in $\beta$, hence $g( \beta) \neq g_1^k(\beta)$ for any $ g_1(x) \in \overline{\mathbb{F}} (x)$. Applying Lemma \ref{weil}, we get $ \mathcal{B} \ll P_1^{1/2 + \varepsilon}$. We record this result in the following lemma:
\begin{lemma} \label{lemma character sum in thm 2 }
Let $ \mathcal{B}$ be as given in equation \eqref{character sum in thm 2}. We have $ \mathcal{B} \ll P_1^{1/2 + \varepsilon}$. 
\end{lemma}

Substituting the bound for $T$ and $ \mathcal{B}$  in equation \eqref{s hat q} we obtain


\begin{align}
\hat{S}_0 (q) &=  \mathop{ \sum_{m \ll \frac{P_2}{Q}  }  \sum_{m \ll \frac{P_2}{Q}  }  }_{  \overline{m^\prime  } \equiv - \overline{ m } -   n \overline{P_2} ( \textrm{mod} \ q)}  \chi_2 ( \overline{m}) \overline{\chi_2 ( \overline{m^\prime})}  \frac{\mathfrak{I} (x , q, m)}{a}  \frac{ \mathfrak{I} (x , q, m^\prime)}{a^\prime} \sideset{}{^*} \sum_{ \beta (P_1)}   \chi_1 (\beta) \overline{\chi_1 (\alpha^\prime)}   \notag \\
 & \hspace{1cm} \times L \sum_{ n\ll  (P_1 q N^\varepsilon)/L   }  \int_{\mathbb{R}} U(y) e \left( - \frac{nL y}{P_1 q}\right) dy  := \hat{S}_0 (q, D) + \hat{S}_0 (q, ND), 
\end{align} 
where $ \hat{S}_0 (q, D) $ (respectively $ \hat{S}_0 (q, ND) $) denotes the  contribution of the diagonal terms (respectively the off-diagonal terms), and $\beta$ and $\alpha^\prime$ are related by equation \eqref{relation beta alpha prime in thm 1}.  Using the bound for the character sum given in Lemma \ref{lemma character sum in thm 2 }, we have that the contribution of the off-diagonal terms is given by

\begin{align} \label{non diagonal in thm 1}
\hat{S}_0 (q, ND) &\ll \sum_{ n\ll  (P_1 q P^\varepsilon)/L   }   \mathop{ \sum_{m \ll \frac{P_2}{Q}  }  \sum_{m \ll \frac{P_2}{Q}  }  }_{  \overline{m^\prime  } \equiv - \overline{ m } -   n \overline{P_2} ( \textrm{mod} \ q)}  \left|\chi_2 ( \overline{m}) \overline{\chi_2 ( \overline{m^\prime})}  \frac{\mathfrak{I} (x , q, m) }{a} \frac{ \mathfrak{I} (x , q, m^\prime)}{a^\prime}  \right|   \notag\\
 & \hspace{1cm} \times P_1^{1/2 + \varepsilon}  L  \left| \int_{\mathbb{R}} U(y) e \left( - \frac{nL y}{P_1 q}\right) dy \right|  \notag\\
 & \ll P^\varepsilon \frac{P_1 q}{L} \ \frac{P_2}{ Q} \frac{P_2}{ Q} \frac{1}{ q} \  \frac{1}{a} \frac{1}{a^\prime}  P_1^{1/2 } L  \ll P^\varepsilon \frac{P_2^2  P_1^{3/2 }}{ Q^4} , 
\end{align}
 as $a, a^\prime \asymp Q$. Also the contribution of the diagonal terms($q = a^\prime$, $m= m^\prime$) is bounded by

\begin{align} \label{diagonal in thm 1}
\hat{S}_0 (q, D) &\ll \sum_{ n\ll  (P_1 q P^\varepsilon)/L   } \sum_{ \substack{m \ll \frac{P_2}{Q} \\ 2 \overline{ m } \equiv  -   n \overline{P_2} ( \textrm{mod} \ q) }} \left|\chi_2 ( \overline{m}) \right|^2 \left| \frac{\mathfrak{I} (x , q, m)}{a}   \right|^2 \sum_{\beta(P_1)} |\chi_1 (\beta)|^2  \notag\\
 & \hspace{1cm} \times L  \left| \int_{\mathbb{R}} U(y) e \left( - \frac{nL y}{P_1 q}\right) dy \right|  \notag\\
 & \ll P^\varepsilon \frac{P_1 q}{L} \ \frac{P_2}{Q}  \frac{1}{q}  \frac{1}{Q^2} P_1 L  \ll P^\varepsilon \frac{ P_1^2 P_2}{Q^3}, 
\end{align} as $a \asymp Q$. Substituting the bounds for the off-diagonal and the diagonal terms given in equations \eqref{non diagonal in thm 1} and \eqref{diagonal in thm 1} we obtain

\begin{align*}
\hat{S}_0 (q) \ll P^\varepsilon \left( \frac{ P_1^2 P_2}{Q^3}   + \frac{P_2^2  P_1^{3/2 }}{ Q^4} \right).
\end{align*}

Substituting the  value of $C_0$ and the bound for $ \hat{S}_0 (q) $ in equation \eqref{S plus before second cauchy} we obtain

\begin{align*}
S_3^+(N)&\ll  P^\varepsilon \frac{N}{  P_1 P_2}   \left|\frac{\tau(\chi_2)}{ \tau ( \overline{\chi_1} )}  \right|  \frac{N}{P_1 } \frac{P_1}{\sqrt{N}}   \int_0^1  \sum_{1\leq q\leq Q}  \frac{1}{q} \sum_{\substack{ L \ll P_1 \\ L \textrm{-} \textrm{dyadic}}}  \left\lbrace\frac{ P_1^2 P_2}{Q^3}  +  \frac{P_2^2  P_1^{3/2 }}{ Q^4} \right\rbrace^{1/2} \\
& \ll  P^{\varepsilon} \frac{N^{3/2}}{ P_1^{3/2} P_2^{1/2}} \left( \frac{P_1 P_2^{1/2} }{ Q^{3/2} } + \frac{P_2 P_1^{3/4}}{Q^2} \right) \ll P^\varepsilon N^{\frac{1}{2} } \left( \frac{N}{P_1^{1/2}} \frac{P_1^{3/4}}{N^{3/4}} + \frac{N P_2^{1/2}}{P_1^{3/4}} \frac{P_1}{N}  \right) \\
& \ll  P^\varepsilon N^{\frac{1}{2} } \left( P^{1/4} P_1^{1/4} + P^{1/2} P_1^{ - 1/4} \right), 
\end{align*}  as  $N \ll P= P_1 P_2$ and $ Q = \left( N/P_1 \right)^{1/2}$.

{\bf Acknowledgement:}  Authors would also like to thank Stat-Math unit, Indian Statistical Institute, Kolkata for wonderful academic atmosphere.  During the work, second author was supported by the Department of Atomic Energy, Government of India, NBHM post doctoral fellowship no: 2/40(15)/2016/R$\&$D-II/5765.

%
%
%
%
%
%
%
%
{}
\end{document}